\documentclass[a4paper,11pt]{article}
\usepackage{amsmath,amsthm,amssymb,enumitem,xcolor}
\usepackage{tikz}

\usepackage[hang]{footmisc}
\usepackage[nosort,nocompress,noadjust]{cite}

\usepackage[bookmarks=false,hyperfootnotes=false,colorlinks,
    linkcolor={red!60!black},
    citecolor={blue!50!black},
    urlcolor={blue!80!black}]{hyperref}

\renewcommand{\eqref}[1]{\hyperref[#1]{(\ref{#1})}}

\newlist{enumlist}{enumerate}{1}
\setlist[enumlist]{labelindent=0cm,label=\arabic*.,ref=\arabic*,labelwidth=2.5ex,labelsep=0.5ex,leftmargin=3ex,align=left,topsep=0.5ex,itemsep=1ex,parsep=1ex}

\newlist{itemlist}{itemize}{1}
\setlist[itemlist]{labelindent=0cm,label=$\bullet$,labelwidth=2.5ex,labelsep=0.5ex,leftmargin=3ex,align=left,topsep=0.5ex,itemsep=1ex,parsep=1ex}

\numberwithin{equation}{section}

{\theoremstyle{definition}\newtheorem{definition}{Definition}[section]
\newtheorem*{definition*}{Definition}

\newtheorem{remark}[definition]{Remark}
\newtheorem{example}[definition]{Example}
\newtheorem*{example*}{Example}
\newtheorem*{examples*}{Examples}}

\newtheorem{proposition}[definition]{Proposition}
\newtheorem{lemma}[definition]{Lemma}
\newtheorem{theorem}[definition]{Theorem}
\newtheorem{corollary}[definition]{Corollary}

{\theoremstyle{definition}}

\newcommand{\SL}{\operatorname{SL}}
\newcommand{\Sd}{\operatorname{Sd}}

\newcommand{\C}{\mathbb{C}}

\newcommand{\eps}{\varepsilon}
\newcommand{\al}{\alpha}
\newcommand{\be}{\beta}
\newcommand{\ot}{\otimes}
\newcommand{\recht}{\rightarrow}
\newcommand{\Z}{\mathbb{Z}}
\newcommand{\vphi}{\varphi}

\newcommand{\id}{\mathord{\text{\rm id}}}
\newcommand{\om}{\omega}

\newcommand{\ovt}{\mathbin{\overline{\otimes}}}

\newcommand{\real}{\operatorname{Re}}

\newcommand{\si}{\sigma}
\newcommand{\R}{\mathbb{R}}
\newcommand{\F}{\mathbb{F}}
\newcommand{\cH}{\mathcal{H}}

\newcommand{\cK}{\mathcal{K}}

\newcommand{\cF}{\mathcal{F}}

\newcommand{\actson}{\curvearrowright}

\newcommand{\cB}{\mathcal{B}}

\newcommand{\cU}{\mathcal{U}}

\newcommand{\cR}{\mathcal{R}}
\newcommand{\dis}{\displaystyle}

\newcommand{\cV}{\mathcal{V}}

\newcommand{\cS}{\mathcal{S}}

\begin{document}

\begin{center}
{\boldmath\Large\bf Bernoulli actions of type III$_1$ and $L^2$-cohomology}

\bigskip

{\sc by Stefaan Vaes$^1$ and Jonas Wahl\footnote{\noindent KU~Leuven, Department of Mathematics, Leuven (Belgium).\\ E-mails: stefaan.vaes@kuleuven.be and jonas.wahl@kuleuven.be. SV and JW are supported by European Research Council Consolidator Grant 614195 RIGIDITY, and by long term structural funding~-- Methusalem grant of the Flemish Government.}}

\end{center}

\begin{abstract}\noindent
We conjecture that a countable group $G$ admits a nonsingular Bernoulli action of type III$_1$ if and only if the first $L^2$-cohomology of $G$ is nonzero. We prove this conjecture for all groups that admit at least one element of infinite order. We also give numerous explicit examples of type III$_1$ Bernoulli actions of the groups $\Z$ and the free groups $\F_n$, with different degrees of ergodicity.
\end{abstract}

\section{Introduction}

Among the most well studied probability measure preserving actions of a countable group $G$ are the \emph{Bernoulli actions} on product spaces $(X_0^G,\mu_0^G)$ given by $(g \cdot x)_h = x_{g^{-1} h}$. Replacing $\mu_0^G$ by an arbitrary product probability measure $\mu = \prod_{g \in G} \mu_g$, using Kakutani's criterion on the equivalence of product measures \cite{Ka48}, it is easy to see when the resulting action on $(X,\mu)$ is \emph{nonsingular}, i.e.\ the action preserves the measure class of $\mu$. However, it turned out to be difficult, even when $G = \Z$, to give criteria when $G \actson (X,\mu)$ is ergodic and to determine its type. Only quite recently, in \cite{Ko09,Ko10,Ko12,DL16}, the first examples of nonsingular type III$_1$ Bernoulli actions of $G = \Z$ were constructed, using an inductive procedure to define $\mu_n$, $n \in \Z$.

We make a systematic study of nonsingular Bernoulli actions $G \actson (X,\mu) = \prod_{g \in G} (X_0,\mu_g)$ of arbitrary countable groups. We conjecture that $G$ admits an ergodic nonsingular Bernoulli action of type III$_1$ if and only if the first $L^2$-cohomology $H^1(G,\ell^2(G))$ is nonzero. It is indeed quite straightforward to see that if $H^1(G,\ell^2(G))=\{0\}$, then $\mu$ is equivalent to a $G$-invariant probability measure of the form $\mu_0^G$, see Theorem \ref{thm.all-II1}. The converse implication turns out to be much more involved. While every non-inner $1$-cocycle $c : G \recht \ell^2(G)$ gives rise to a nonsingular Bernoulli action $G \actson (X,\mu)$, the ergodicity and type of this Bernoulli action depend in a very subtle way on the behavior of the $1$-cocycle $c$.

The $L^2$-cohomology $H^1(G,\ell^2(G))$ can be nonzero for two reasons: when $G$ has positive first $L^2$-Betti number $\beta_1^{(2)}(G) > 0$ and when $G$ is an infinite amenable group. We therefore prove the conjecture in the following two separate cases:
\begin{enumlist}
\item when $G$ has positive first $L^2$-Betti number $\beta_1^{(2)}(G)> 0$ and $G$ contains an infinite subgroup $\Lambda < G$ such that $\beta_1^{(2)}(\Lambda) <  \beta_1^{(2)}(G)$, see Theorem \ref{thm.positive-betti-III1};
\item when $G$ is an amenable group, see Theorem \ref{thm.amenable-III1}.
\end{enumlist}
Since $\beta_1^{(2)}(\Z) = 0$, these two statements imply that our conjecture holds when $G$ contains an element of infinite order.

A crucial ingredient to prove the first statement above is a non-inner $1$-cocycle $c : G \recht \ell^2(G)$ that vanishes on an infinite subgroup $\Lambda < G$. Such a $1$-cocycle does not exist when $G$ is amenable. More precisely, when $G$ is infinite and amenable, by \cite[Theorem 2.5]{PT10}, every non-inner $1$-cocycle $c : G \recht \ell^2(G)$ is proper and therefore does not vanish on any infinite subset. When $c : G \recht \ell^2(G)$ is a proper $1$-cocycle, the ergodicity and type of the associated nonsingular Bernoulli action depend subtly on the growth of the cocycle, i.e.\ the growth of the function $g \mapsto \|c_g\|_2$. The main issue is that if $\|c_g\|_2$ grows too fast, then $G \actson (X,\mu)$ is \emph{dissipative}. Recall here that a nonsingular action $G \actson (X,\mu)$ is called dissipative if there exists a Borel set $\cU \subset X$ such that all $g \cdot \cU$, $g \in G$, are disjoint and $\bigcup_{g \in G} g \cdot \cU = X$, up to measure zero. On the other hand, $G \actson (X,\mu)$ is called \emph{conservative} if for every non-null Borel set $\cU \subset X$, there exists a $g \in G \setminus \{e\}$ such that $g \cdot \cU \cap \cU$ is non-null. A Borel set $\cU \subset X$ is called \emph{wandering} if all $g \cdot \cU$, $g \in G$, are disjoint. A nonsingular action $G \actson (X,\mu)$ is conservative if and only if every wandering set has measure zero.
In Proposition \ref{prop.conservative}, we provide a quite sharp, quantitative conservative/dissipative criterion for nonsingular Bernoulli actions in terms of the growth of the associated $1$-cocycle, thus answering \cite[Question 10.5]{DL16}.

We then prove that an amenable group $G$ admits $1$-cocycles $c : G \recht \ell^2(G)$ of arbitrarily slow growth, see Proposition \ref{prop.cocycle-small-growth}. This result is analogous to \cite[Proposition 3.10]{CTV05}, where it is shown that a group with the Haagerup admits proper $1$-cocycles of arbitrarily slow growth into \emph{some} unitary representation. By combining Proposition \ref{prop.cocycle-small-growth} with the conservativeness criterion in Proposition \ref{prop.conservative}, we construct ergodic type III$_1$ Bernoulli actions for all infinite amenable groups $G$, thus proving the second statement above.

For each of the groups $G$ in the two statements above, we actually construct nonsingular Bernoulli actions $G \actson (X,\mu)$ that are \emph{weakly mixing} and of \emph{stable type III$_1$} in the sense of \cite[Section 1.3]{BN11}, meaning that for every ergodic probability measure preserving (pmp) action $G \actson (Y,\eta)$, the diagonal action $G \actson (Y \times X,\eta \times \mu)$ remains ergodic and of type III$_1$.

As a consequence of our methods, we also give explicit examples of type III$_1$ Bernoulli actions of $\Z$ in Corollaries \ref{cor.explicit-Z} and \ref{cor.explicit-Z-power-dissipative}, complementing the less explicit inductive constructions in \cite{Ko09,Ko10,Ko12,DL16}. For some of these examples of Bernoulli shifts $T$, all powers $T \times \cdots \times T$ remain ergodic and of type III$_1$ (as in the examples in \cite{Ko10}), but others admit a power that is dissipative~--~and such examples were not available so far.

In the final Section \ref{sec.bernoulli-free-group}, we give several concrete examples of nonsingular Bernoulli actions $G \actson (X,\mu)$ of the free groups $G = \F_n$.
\begin{itemlist}
\item In Example \ref{ex.free-groups}, we construct nonsingular Bernoulli actions of $\F_n$, $n \geq 2$, that are of type III$_\lambda$ for arbitrary $\lambda \in (0,1)$. It is unknown whether such actions exist for $G = \Z$. So far, all available constructions of type III Bernoulli actions of $\Z$, including ours, rely on the assumption that all $\mu_n$, $n < 0$, are identical. Under that hypothesis, it is proven in \cite[Corollary 3.3]{DL16} that all nonsingular Bernoulli actions are of type I (the dissipative case), type II$_1$ or type III$_1$, but never of type III$_\lambda$.

\item In Proposition \ref{prop.free-group-prescribed-invariants}, we give examples of nonsingular Bernoulli actions of $\F_n$, $n \geq 3$, that are \emph{strongly ergodic}. Moreover, the Connes invariants of the associated orbit equivalence relation (see \cite{Co74,HMV17}) can take any prescribed value: in Example \ref{ex.free-groups}, we provide almost periodic examples whose $\Sd$-invariant is any countable dense subgroup of $\R_*^+$ and we provide non almost periodic examples for which the $\tau$-invariant is an arbitrary topology on $\R$ induced by a unitary representation of $\R$. This answers \cite[Problem 3]{HMV17}.

\item In Proposition \ref{prop.stable-type-concrete-examples} and Example \ref{ex.all-values-a-b}, we construct nonsingular, weakly mixing Bernoulli actions $\F_n \actson (X,\mu)$, $n \geq 2$, with a variety of stable types. This includes examples of stable type III$_\lambda$, i.e.\ such that for every ergodic pmp action $\F_n \actson (Y,\eta)$, the diagonal action $\F_n \actson (Y \times X,\eta \times \mu)$ is of type III$_\lambda$, but it also includes examples where the type of these diagonal actions ranges over III$_\mu$ with $\mu \in \{1\} \cup \{\lambda^{1/k} \mid k \geq 1\}$, for any fixed $0 < \lambda < 1$.

\item In Proposition \ref{prop.concrete-free-group-1}, we give examples of type III$_1$ nonsingular Bernoulli actions $G \actson (X,\mu)$ of $G = \F_2$ associated with a proper $1$-cocycle $c : G \recht \ell^2(G)$ such that the $m$-fold diagonal action $G \actson (X^m,\mu^m)$ is dissipative for $m$ large enough. Finally, we give examples of dissipative Bernoulli actions $\F_2 \actson (X,\mu)$ of the free group $\F_2$ in Proposition \ref{prop.free-group-dissipative}.
\end{itemlist}

We conclude our discussion of stable types by providing a positive answer to \cite[Question 4.6]{BN11} in Proposition \ref{prop.BN-question}, where we prove that for every countable infinite group $G$ and every $\lambda \in (0,1]$, there exists an essentially free nonsingular action $G \actson (X,\mu)$ that is amenable, weakly mixing and of stable type III$_\lambda$.

{\bf Acknowledgment.} We thank Zemer Kosloff for his suggestion to also consider the stable type of our nonsingular Bernoulli actions, in connection with \cite{BN11}. SV thanks the Isaac Newton Institute for Mathematical Sciences for support and hospitality during the programme {\it Operator Algebras: Subfactors and their Applications} when work on this paper was undertaken, supported by EPSRC Grant Number EP/K032208/1.

\section{Preliminaries and notations}

Let $G$ be a countable infinite group. Given $0 < \mu_g(0) < 1$ for all $g \in G$, we consider the product probability space
$$(X,\mu) = \prod_{g \in G} (\{0,1\},\mu_g)$$
and the Bernoulli action $G \actson X$ given by $(g \cdot x)_k = x_{g^{-1} k}$ for all $g, k \in G$, $x \in X$. By Kakutani's theorem \cite{Ka48} on the equivalence of product measures, we get that the action $G \actson (X,\mu)$ is nonsingular if and only if for every $g \in G$, we have that
\begin{equation}\label{eq.nonsingular}
\sum_{k \in G} \Bigl( \sqrt{\mu_{gk}(0)} - \sqrt{\mu_k(0)}\Bigr)^2 + \sum_{k \in G} \Bigl( \sqrt{\mu_{gk}(1)} - \sqrt{\mu_k(1)}\Bigr)^2 < \infty \; .
\end{equation}
Note that $\mu$ is nonatomic if and only if
\begin{equation}\label{eq.nonatomic}
\sum_{k \in G} \min\{\mu_k(0),\mu_k(1)\} = +\infty
\end{equation}
and we always make this assumption. Also note that if there exists a $\delta > 0$ such that $\delta \leq \mu_k(0) \leq 1-\delta$ for all $k \in G$, then the nonsingularity condition \eqref{eq.nonsingular} is equivalent with the condition
\begin{equation}\label{eq.easier-nonsingular}
\sum_{k \in G} (\mu_{gk}(0) - \mu_k(0))^2 < \infty
\end{equation}
for every $g \in G$, see \cite{Ka48}.

When proving that certain nonsingular Bernoulli actions $G \actson (X,\mu)$ are of type III$_1$, it is often useful to restrict the action of $G$ to a subgroup $\Lambda < G$. We therefore fix the following general framework: a countable infinite group $\Lambda$ acting freely on a countable set $I$. Given any function $F : I \recht (0,1)$, we define the product probability space $(X,\mu) = \prod_{i \in I} (\{0,1\},\mu_i)$ where $\mu_i(0) = F(i)$ and we consider the Bernoulli action $\Lambda \actson X$ given by $(g \cdot x)_i = x_{g^{-1} \cdot i}$. We always make the following two assumptions:
\begin{equation}\label{eq.conds-F}
\begin{split}
& \text{there exists a $\delta > 0$ such that}\;\; \delta \leq F(i) \leq 1-\delta \;\; \text{for all $i \in I$,}\\
& \text{for every $g \in G$, we have that}\;\; \sum_{i \in I} (F(g \cdot i) - F(i))^2 < \infty \; .
\end{split}
\end{equation}
Then, the action $\Lambda \actson (X,\mu)$ is nonsingular and essentially free. The Radon-Nikodym cocycle $\om : \Lambda \times X \recht (0,\infty)$ is defined by
\begin{equation}\label{eq.general-RN}
\int_X F(x) \, \om(g,x) \, d\mu(x) = \int_X F(g^{-1} \cdot x) \, d\mu(x)
\end{equation}
for all positive Borel functions $F : X \recht [0,+\infty)$ and all $g \in \Lambda$. Given any enumeration $I = \{i_1,i_2,\ldots\}$, we have that
\begin{equation}\label{eq.RN-conv-ae}
\om(g,x) = \lim_n \prod_{k=1}^n \frac{\mu_{g \cdot i_k}(x_{i_k})}{\mu_{i_k}(x_{i_k})} \quad\text{for a.e.\ $x \in X$.}
\end{equation}
The \emph{Maharam extension} of $\Lambda \actson (X,\mu)$ is the (infinite) measure preserving action
\begin{equation}\label{eq.maharam}
\Lambda \actson (X \times \R, \mu \times \nu) : g \cdot (x,s) = (g \cdot x, \log(\om(g,x)) + s) \;\; , \;\; d\nu(s) = \exp(-s) \, ds \; .
\end{equation}
The Maharam extension $\Lambda \actson X \times \R$ commutes with the translation action $\R \actson X \times \R$ given by $t \cdot (x,s) = (x,s+t)$. Identifying the algebra of $\Lambda$-invariant elements $L^\infty(X \times \R)^\Lambda$ with $L^\infty(Z,\rho)$ for some standard probability space $(Z,\rho)$, we thus find a nonsingular action $\R \actson (Z,\rho)$.
Assuming that $\Lambda \actson (X,\mu)$ is nonsingular, essentially free and ergodic, its type can be determined as follows in terms of $\R \actson (Z,\rho)$~: if the action $\R \actson Z$ is measurably conjugate with the translation action $\R \actson \R$, we get type I or II (the semifinite case); if the action is conjugate with $\R \actson \R / \log(\lambda) \Z$ for $0 < \lambda < 1$, we get type III$_\lambda$; if the action is the trivial action on one point (i.e.\ the Maharam extension is ergodic), we get type III$_1$; and finally, if the action is properly ergodic, we get type III$_0$.

Note that by \eqref{eq.conds-F}, we can associate with $F : I \recht (0,1)$ the $1$-cocycle
\begin{equation}\label{eq.one-cocycle-c}
c : \Lambda \recht \ell^2(I) : c_g(i) = F(i) - F(g^{-1} \cdot i) \; .
\end{equation}

Recall that a nonsingular action $G \actson (X,\mu)$ is called \emph{weakly mixing} if for every ergodic probability measure preserving (pmp) action $G \actson (Y,\eta)$, the diagonal action $G \actson (Y \times X,\eta \times \mu)$ is ergodic. Following \cite[Section 1.3]{BN11}, an essentially free, nonsingular action $G \actson (X,\mu)$ is said to be of \emph{stable type III$_1$} if for every pmp action $G \actson (Y,\eta)$, the diagonal action $G \actson (Y \times X,\eta \times \mu)$ is of type III$_1$. So $G \actson (X,\mu)$ is of stable type III$_1$ if and only if for every pmp action $G \actson (Y,\eta)$ and using the Maharam extension, we have $L^\infty(Y \times X \times \R)^G = L^\infty(Y \times X)^G \ot 1$. In particular, $G \actson (X,\mu)$ is weakly mixing and of stable type III$_1$ if and only if the Maharam extension $G \actson X \times \R$ is weakly mixing.

Let $G$ be a countable group. The \emph{amenability} of an essentially free nonsingular action $G \actson (X,\mu)$ was defined in \cite[Definition 1.4]{Zi76a} through a fixed point property. When $\mu$ is an invariant probability measure, this notion is equivalent with the amenability of $G$. In general, this notion is equivalent with the injectivity of the crossed product von Neumann algebra $L^\infty(X) \rtimes G$ by \cite{Zi76b} and \cite[Theorem 2.1]{Zi76c}. Denote by $\lambda : G \recht \cU(\ell^2(G))$ the left regular representation. By \cite[Theorem 3.1.6]{AD01}, the amenability of $G \actson (X,\mu)$ is equivalent with the existence of a sequence of Borel maps $\xi_n : X \recht \ell^2(G)$ with the following properties: for all $n$ and a.e.\ $x \in X$, we have that $\|\xi_n(x)\|_2 = 1$; and for all $g \in G$ and $P \in L^1(X,\mu)$, we have that
$$\lim_n \int_X \langle \lambda_g \xi_n(g^{-1} \cdot x),\xi_n(x) \rangle \, P(x) \, d\mu(x) = \int_X P(x) \, d\mu(x) \; .$$

\section{\boldmath Groups with trivial first $L^2$-cohomology}

The following theorem says that for groups with vanishing first $L^2$-cohomology, a nonsingular Bernoulli action is either probability measure preserving (pmp) or dissipative, and thus, never of type III.

\begin{theorem}\label{thm.all-II1}
Let $G$ be a countable infinite group with $H^1(G,\ell^2(G)) = \{0\}$. Assume that $(\mu_g)_{g \in G}$ is a family of probability measures on a standard Borel space $X_0$. If the Bernoulli action $G \actson (X,\mu) = \prod_{g \in G} (X_0,\mu_g)$ is nonsingular, then there exists a partition $X_0 = Y_0 \sqcup Z_0$ into Borel sets such that, writing $Y = Y_0^G \subset X$, we have
\begin{enumlist}
\item $\mu(Y) > 0$ and $\mu|_Y \sim \nu^G$ for some probability measure $\nu$ on $Y_0$, so that $G \actson (Y,\mu)$ is an ergodic pmp Bernoulli action;
\item $\sum_{g \in G} \mu_g(Z_0) < \infty$, so that the action $G \actson (X \setminus Y,\mu)$ is dissipative.
\end{enumlist}
\end{theorem}

Note that there are large classes of groups for which $H^1(G,\ell^2(G)) = \{0\}$, so that all their ergodic nonsingular Bernoulli actions must be of type II$_1$ or type I. This holds in particular for all infinite groups with property~(T), for all nonamenable groups that admit an infinite amenable normal subgroup, and for all direct product groups $G = G_1 \times G_2$ with $G_1$ infinite and $G_2$ nonamenable.

\begin{proof}
Since $G \actson (X,\mu)$ is nonsingular, all measures $\mu_g$ are in the same measure class. We fix a probability measure $\mu_0$ on $X_0$ such that $\mu_g \sim \mu_0$ for all $g \in G$. Define the unit vectors $\xi_g \in L^2(X_0,\mu_0)$ given by $\xi_g = \sqrt{d\mu_g / d\mu_0}$. By Kakutani's \cite{Ka48}, we get that $\sum_{k \in G} \|\xi_{gk} - \xi_k\|_2^2 < \infty$ for all $g \in G$. So, the map
$$c : G \recht \ell^2(G) \ot L^2(X_0,\mu_0) : c_g = \sum_{k \in G} \delta_k \ot (\xi_k - \xi_{g^{-1} k})$$
is a well defined $1$-cocycle.

Write $\cK = L^2(X_0,\mu_0)$. Since $H^1(G,\ell^2(G)) = \{0\}$ and $G$ is infinite, the group $G$ is non\-amenable. Since $G$ is non\-amenable and the representation of $G$ on $\ell^2(G) \ot \cK$ is a multiple of the regular representation, this representation has no almost invariant unit vectors. It follows that the inner $1$-cocycles form a closed subspace of the space of $1$-cocycles $Z^1(G,\ell^2(G) \ot \cK)$ equipped with the topology of pointwise convergence. Fix a sequence of finite rank projections $P_n$ on $\cK$ that converge to $1$ strongly. Since $H^1(G,\ell^2(G)) = \{0\}$, every $g \mapsto (1 \ot P_n) c_g$ is an inner $1$-cocycle. Since $\lim_n (1 \ot P_n)c_g = c_g$ for every $g \in G$, it then follows that also $c$ is inner. This means that there exists a $\xi_0 \in \cK$ such that
\begin{equation}\label{eq.inner-xi0}
\sum_{k \in G} \|\xi_k - \xi_0\|_2^2 < \infty \; .
\end{equation}
We get in particular that $\xi_k \recht \xi_0$ as $k \recht \infty$ in $G$. So, $\xi_0$ is positive a.e.\ and $\|\xi_0\|_2 = 1$. Denote by $\nu$ the unique probability measure on $X_0$ such that $\nu \prec \mu_0$ and $\xi_0 = \sqrt{d\nu / d\mu_0}$. Write $X_0 = Y_0 \sqcup Z_0$ such that $\nu(Z_0) = 0$ and $\nu \sim \mu_0|_{Y_0}$.

Since $\nu(Z_0) = 0$, we have
$$\|\xi_k - \xi_0\|_2^2 \geq \int_{Z_0} \xi_k(x)^2 \, d\mu_0(x) = \mu_k(Z_0) \; .$$
It follows that $\sum_{k \in G} \mu_k(Z_0) < \infty$. Writing $Y = Y_0^G \subset X$, we conclude that $\mu(Y) > 0$. Since $\mu_k|_{Y_0} \sim \nu$ for all $k \in G$, it follows from \eqref{eq.inner-xi0} and \cite{Ka48} that $\mu|_Y \sim \nu^G$.

Write $Z = \{x \in X \mid x_e \in Z_0\}$. It follows that
$$\sum_{k \in G} \mu(k \cdot Z) = \sum_{k \in G} \mu_k(Z_0) < \infty \; .$$
Since $X \setminus Y = \bigcup_{k \in G} k \cdot Z$, it follows that the action $G \actson (X \setminus Y,\mu)$ is dissipative.
\end{proof}

\section{A criterion for conservativeness}\label{sec.conservative}

Recall that a nonsingular essentially free action $\Lambda \actson (X,\mu)$ is called \emph{conservative} if there is no nonnegligible Borel set $A \subset X$ such that all $g \cdot A$, $g \in \Lambda$ are disjoint. Note that $\Lambda \actson (X,\mu)$ is conservative if and only if the orbit equivalence relation has no type I direct summand, which is in turn equivalent to the crossed product $L^\infty(X) \rtimes \Lambda$ having no type I direct summand. So, using e.g.\ \cite[Theorem XII.1.1]{Ta03}, a nonsingular essentially free action $\Lambda \actson (X,\mu)$ is conservative if and only if its Maharam extension given by \eqref{eq.maharam} is conservative.


The key ingredient to prove Theorems \ref{thm.positive-betti-III1} and \ref{thm.amenable-III1} is the following criterion to ensure that a Bernoulli action is conservative. The criterion says that it suffices that the $1$-cocycle $c$ given by \eqref{eq.one-cocycle-c} has logarithmic growth in at least one direction, thus providing an answer to \cite[Question 10.5]{DL16}. The second point of the proposition is easier and is a straightforward generalization of \cite[Lemma 2.2]{Ko12} to Bernoulli actions of arbitrary countable groups.

\begin{proposition}\label{prop.conservative}
Let $\Lambda \actson I$ be a free action of the countable group $\Lambda$ on the countable set $I$ and let $F : I \recht (0,1)$ be a function satisfying \eqref{eq.conds-F}, in particular $\delta \leq F(i) \leq 1-\delta$ for all $i \in I$. Denote by $\Lambda \actson (X,\mu)$ the associated Bernoulli action and by $c : \Lambda \recht \ell^2(I)$ the associated $1$-cocycle as in \eqref{eq.one-cocycle-c}.

\begin{enumlist}
\item If $\sum_{g \in \Lambda} \exp(-\kappa \|c_g\|_2^2) = +\infty$ for some $\kappa > \delta^{-2} + \delta^{-1}(1-\delta)^{-2}$, then the action $\Lambda \actson (X,\mu)$ is conservative.

\item If $\sum_{g \in \Lambda} \exp(- \frac{1}{2} \|c_g\|_2^2) < +\infty$, then the action $\Lambda \actson (X,\mu)$ is dissipative.
\end{enumlist}

In particular, if $1/3 \leq F(i) \leq 2/3$ for all $i \in I$ and if $\sum_{g \in \Lambda} \exp(-16 \|c_g\|_2^2) = +\infty$, then the action $\Lambda \actson (X,\mu)$ is conservative.
\end{proposition}

\begin{proof}
Denote by $\om : \Lambda \times X \recht (0,+\infty)$ the Radon-Nikodym cocycle given by \eqref{eq.general-RN}. By \cite[Proposition 1.3.1]{Aa97}, whose proof remains valid for arbitrary groups $\Lambda$, an essentially free nonsingular action $\Lambda \actson (X,\mu)$ is conservative if and only $\sum_{g \in \Lambda} \om(g,x) = +\infty$ for a.e.\ $x \in X$, while it is dissipative if and only if $\sum_{g \in \Lambda} \om(g,x) < +\infty$ for a.e.\ $x \in X$.

Write $\kappa_0 = \delta^{-2} + \delta^{-1}(1-\delta)^{-2}$. We start by proving that
\begin{equation}\label{eq.estim-om-negative-power}
\int_X \om(g,x)^{-2} \, d\mu(x) \leq \exp(\kappa_0 \, \|c_g\|_2^2) \quad\text{for all $g \in \Lambda$.}
\end{equation}
To prove \eqref{eq.estim-om-negative-power}, not that for all $0 < a,b < 1$,
$$\frac{a^3}{b^2} + \frac{(1-a)^3}{(1-b)^2} = 1 + \frac{a+2b-2ab-b^2}{b^2 \, (1-b)^2} \, (a-b)^2$$
and that
$$0 \leq \frac{a+2b-2ab-b^2}{b^2 \, (1-b)^2} \leq \kappa_0 \quad\text{for all}\;\; \delta \leq a,b \leq 1-\delta \; .$$
Fix an enumeration $I = \{i_1,i_2,\ldots\}$ and define the functions
$$\om_n : \Lambda \times X \recht (0,+\infty) : \om_n(g,x) = \prod_{k=1}^n \frac{\mu_{g \cdot i_k}(x_{i_k})}{\mu_{i_k}(x_{i_k})} \; .$$
Fix $g \in \Lambda$. By \eqref{eq.RN-conv-ae}, we have that $\om_n(g,x) \recht \om(g,x)$ for all $g \in \Lambda$ and a.e.\ $x \in X$. By Fatou's lemma, we get that
\begin{align*}
\int_X \om(g,x)^{-2} \, d\mu(x) & \leq \liminf_n \int_X \om_n(g,x)^{-2} \, d\mu(x) \\
& = \liminf_n \prod_{k = 1}^n \Bigl(  \, \frac{F(i_k)^3}{F(g \cdot i_k)^2} \, + \, \frac{(1-F(i_k))^3}{(1-F(g \cdot i_k))^2} \, \Bigr) \\
& \leq \liminf_n \prod_{k=1}^n \bigl( 1 + \kappa_0 (F(i_k) - F(g \cdot i_k))^2 \bigr) \\
& \leq \liminf_n \exp\Bigl( \kappa_0 \sum_{k=1}^n (F(i_k) - F(g \cdot i_k))^2 \Bigr) \\
& = \exp\bigl( \kappa_0 \|c_g\|_2^2 \bigr) \; .
\end{align*}
So, \eqref{eq.estim-om-negative-power} is proved.

Assume that $\kappa > \kappa_0$ and that $\sum_{g \in \Lambda} \exp(-\kappa \|c_g\|_2^2) = +\infty$. We have to prove that $\Lambda \actson (X,\mu)$ is conservative. Write $\kappa_1 = \frac{1}{2}(\kappa_0 + \kappa)$ and $\kappa_2 = \frac{3}{4}\kappa + \frac{1}{4} \kappa_0$. Note that $\kappa_0 < \kappa_1 < \kappa_2 < \kappa$. We claim that there exists an increasing sequence $s_k \in (0,+\infty)$ such that $\lim_k s_k = +\infty$ and
\begin{equation}\label{eq.claim-sk}
\# \bigl\{ g \in \Lambda \mid \|c_g\|_2^2 \leq s_k \bigr\} \geq \exp( \kappa_2 s_k) \quad\text{for all $k \geq 1$.}
\end{equation}
Define, for every $s \geq 0$,
$$\vphi(s) = \# \bigl\{ g \in \Lambda \mid \|c_g\|_2^2 \leq s \bigr\} \; .$$
Then,
$$+\infty = \frac{1}{\kappa} \sum_{g \in \Lambda} \exp(-\kappa \|c_g\|_2^2) = \int_0^{+\infty} \vphi(s) \, \exp(-\kappa s) \, ds \; .$$
If there exists an $s_0 \geq 0$ such that $\vphi(s) \leq \exp( \kappa_2  s)$ for all $s \geq s_0$, the integral on the right hand side is finite. So such an $s_0$ does not exist and the claim is proven. We fix the sequence $s_k$ as in the claim.

Choose finite subsets $\cF_k \subset \Lambda$ such that $\|c_g\|_2^2 \leq s_k$ for all $g \in \cF_k$ and
$$|\cF_k| \in \bigl[\exp(\kappa_2 s_k) - 1,\exp(\kappa_2 s_k)\bigr] \; .$$
For every $k$ and every $g \in \cF_k$, define
$$\cU_{g,k} = \{x \in X \mid \om(g,x) \leq \exp(-\kappa_1 s_k)\} \; .$$
When $x \in \cU_{g,k}$, we have $\om(g,x)^{-2} \geq \exp(2 \kappa_1 s_k)$. It thus follows from \eqref{eq.estim-om-negative-power} that
$$\mu(\cU_{g,k}) \leq \exp((\kappa_0 - 2 \kappa_1)s_k)$$
for all $k$ and all $g \in \cF_k$. Defining $\cV_k = \bigcup_{g \in \cF_k} \cU_{g,k}$, we get that
$$\mu(\cV_k) \leq \exp((\kappa_2 + \kappa_0 - 2 \kappa_1) s_k) = \exp(- \eps s_k) \; ,$$
where $\eps = (\kappa-\kappa_0)/4 > 0$. So, $\mu(\cV_k) \recht 0$ when $k \recht \infty$.

When $x \in X \setminus \cV_k$, we have $\om(g,x) \geq \exp(-\kappa_1 s_k)$ for all $g \in \cF_k$. Therefore,
$$\sum_{g \in \Lambda} \om(g,x) \geq |\cF_k| \, \exp(-\kappa_1 s_k) \geq \exp((\kappa_2-\kappa_1)s_k) - 1 \; .$$
Since the right hand side tends to infinity as $k \recht \infty$, it follows that $\sum_{g \in \Lambda} \om(g,x) = +\infty$ for a.e.\ $x \in X$. So, $\Lambda \actson (X,\mu)$ is conservative.

%
%

To prove the second statement, we claim that
\begin{equation}\label{eq.estim-sqrt-om}
\int_X \sqrt{\om(g,x)} \, d\mu(x) \leq \exp\bigl(-\frac{1}{2} \, \|c_g\|_2^2 \bigr) \quad\text{for all $g \in \Lambda$.}
\end{equation}
The proof of \eqref{eq.estim-sqrt-om} is identical to the proof of \eqref{eq.estim-om-negative-power}, using that
$$\sqrt{ab} + \sqrt{(1-a)(1-b)} \leq 1 - \frac{1}{2}(b-a)^2 \quad\text{for all $0 \leq a,b \leq 1$.}$$
Assuming that $\sum_{g \in \Lambda} \exp(- \frac{1}{2} \|c_g\|_2^2) < +\infty$, it follows from \eqref{eq.estim-sqrt-om} that
$$\int_X \Bigl( \sum_{g \in \Lambda} \sqrt{\om(g,x)}\Bigr) \, d\mu(x) < +\infty \; .$$
So for a.e.\ $x \in X$, we have $\sum_{g \in \Lambda} \sqrt{\om(g,x)} < +\infty$ and, a fortiori, $\sum_{g \in \Lambda} \om(g,x) < \infty$. So, $\Lambda \actson (X,\mu)$ is dissipative.
\end{proof}

\section{\boldmath Groups with positive first $L^2$-Betti number}

We prove that ``almost all'' groups with positive first $L^2$-Betti number admit a nonsingular Bernoulli action of type III$_1$.

\begin{theorem}\label{thm.positive-betti-III1}
Let $G$ be a countable group with $\beta_1^{(2)}(G) > 0$. Assume that one of the following conditions holds.
\begin{enumlist}
\item\label{cond-G-one} $G$ has at least one element of infinite order.
\item\label{cond-G-two} $G$ admits an infinite amenable subgroup.
\item\label{cond-G-three} $\beta_1^{(2)}(G) \geq 1$.
\item\label{cond-G-four} $G$ is residually finite; or more generally, $G$ admits a finite index subgroup $G_0 < G$ such that $[G:G_0] \geq \beta_1^{(2)}(G)^{-1}$.
\end{enumlist}
Then $G$ satisfies the assumptions of Lemma \ref{lem.construction} below and thus, $G$ admits a nonsingular Bernoulli action that is essentially free, ergodic, of type III$_1$ and nonamenable in the sense of Zimmer and that has a weakly mixing Maharam extension.
\end{theorem}

Although it sounds unlikely that all groups with positive first $L^2$-Betti number satisfy one of the assumptions of Theorem \ref{thm.positive-betti-III1}, we have no explicit counterexample. It is in particular unclear whether all torsion groups constructed in \cite[Theorem 2.3]{Os08} as quotients of $\Z / m\Z * \Z / m \Z$ satisfy condition \eqref{cond-G-four}.

Theorem \ref{thm.positive-betti-III1} is deduced from the following technical lemma that we prove first.

\begin{lemma}\label{lem.construction}
Let $G$ be a countable infinite group. Assume that $G$ admits subgroups $\Lambda < G_0 < G$ such that $\Lambda$ is infinite, $G_0 < G$ has finite index and $\beta_1^{(2)}(\Lambda) < \beta_1^{(2)}(G_0)$. Then $G$ admits a nonsingular Bernoulli action that is essentially free, ergodic, of type III$_1$ and nonamenable in the sense of Zimmer and that has a weakly mixing Maharam extension.
\end{lemma}

\begin{proof}
We first prove the lemma when $\Lambda < G$ is an infinite subgroup with $\beta_1^{(2)}(\Lambda) < \beta_1^{(2)}(G)$, i.e.\ the case where $G_0 = G$. Denote by $\lambda : G \recht \cU(\ell^2(G))$ the left regular representation. Since $\beta_1^{(2)}(G) > 0$, we have that $G$ is nonamenable and we can fix a finite subset $\cF \subset G$ and $\eps_0 > 0$ such that
\begin{equation}\label{eq.choice-gap}
\Bigl\| \sum_{g \in \cF} \lambda_g \Bigr\| \leq (1-\eps_0) |\cF| \; .
\end{equation}

By \cite[Theorem 2.2]{PT10}, we have that $\beta_1^{(2)}(\Lambda)$ equals the $L(G)$-dimension of $H^1(\Lambda,\ell^2(G))$. So, the kernel of the restriction map $H^1(G,\ell^2(G)) \recht H^1(\Lambda,\ell^2(G))$ has positive $L(G)$-dimension. Therefore, we can choose a non-inner $1$-cocycle $b : G \recht \ell^2(G)$ with the property that $b_g = 0$ for all $g \in \Lambda$.

Denote by $H: G \recht \C$ the function given by $H(k) = b_k(k)$ for all $k \in G$. Then, $H(e) = 0$ and
$$b_g(k) = H(k) - H(g^{-1} k) \quad \text{for all}\;\;  g,k \in G \; .$$
Since $b$ vanishes on $\Lambda$, the function $H$ is invariant under left translation by $\Lambda$. Since $b$ is not identically zero, $H$ is not the zero function. Replacing $b$ by $i b$ if needed, we may assume that the real part $\real H$ is not identically zero. At the end of the proof, we explain that the $1$-cocycle $b$ may be chosen so that $\real H$ takes at least three different values.

For any fixed $\kappa_1,\kappa_2 > 0$, we define the function
$$\cF : \R \recht [-\kappa_1,\kappa_2] : \cF(t) = \begin{cases} - \kappa_1 &\;\;\text{if $t \leq -\kappa_1$,}\\
t &\;\;\text{if $-\kappa_1 \leq t \leq \kappa_2$,}\\
\kappa_2 &\;\;\text{if $t \geq \kappa_2$.}\end{cases}$$
Note that $|\cF(t) - \cF(s)| \leq |t-s|$ for all $s,t \in \R$.

We define $K : G \recht [-\kappa_1,\kappa_2] : K(k) = \cF(\real H(k))$. Since $\real H$ takes at least three different values, we can fix $\kappa_1,\kappa_2 > 0$ so that the range of $K$ generates a dense subgroup of $\R$, meaning that there is no $a > 0$ such that $K(k) \in \Z a$ for all $k \in G$. Note that $K$ is invariant under left translation by $\Lambda$.

We then fix $\eps_1 > 0$ such that
$$\exp(\eps_1 \kappa_i) \leq 2 \quad\text{for $i=1,2$, and}\quad \exp\Bigl(-\frac{3}{5} \, \eps_1^2 \, \|b_g\|_2^2\Bigr) > 1 -\eps_0 \quad\text{for all $g \in \cF$.}$$
Define the function
$$F : G \recht [1/3,2/3] : F(k) = \frac{1}{1+\exp(\eps_1 K(k))} \; .$$
Associated with $F$, we have the product probability measure $\mu$ on $X = \{0,1\}^G$ given by $\mu = \prod_{k \in G} \mu_k$ with $\mu_k(0) = F(k)$.

For every $g \in G$, we have that
$$\sum_{k \in G} |F(gk) - F(k)|^2 \leq \eps_1^2 \sum_{k \in G} |K(gk) - K(k)|^2 \leq \eps_1^2 \sum_{k \in G} |H(gk) - H(k)|^2 = \eps_1^2 \, \|b_g\|_2^2 \; .$$
So, the Bernoulli action $G \actson (X,\mu)$ is essentially free, nonsingular and the $1$-cocycle $c : G \recht \ell^2(G)$ given by $c_g(k) = F(k)-F(g^{-1} k)$ satisfies $\|c_g\|_2 \leq \eps_1 \, \|b_g\|_2$ for all $g \in G$.


Denote by $\om : G \times X \recht (0,+\infty)$ the Radon-Nikodym cocycle given by \eqref{eq.general-RN} and consider the Maharam extension $G \actson (X \times \R,\mu \times \nu)$ given by \eqref{eq.maharam}. Let $G \actson (Y,\eta)$ be any pmp action and consider the diagonal action $G \actson (Y \times X \times \R, \eta \times \mu \times \nu)$. We prove that $L^\infty(Y \times X \times \R)^G = L^\infty(Y)^G \ot 1 \ot 1$. Once this statement is proved, it follows that $G \actson (X,\mu)$ is ergodic and of type III$_1$ and that its Maharam extension is weakly mixing.

Since $F$ is invariant under left translation by $\Lambda$, we have that $\om(g,x) = 1$ for all $g \in \Lambda$, $x \in X$ and we have that the action $\Lambda \actson (X,\mu)$ is isomorphic with a probability measure preserving Bernoulli action of $\Lambda$. So, a $G$-invariant function $Q \in L^\infty(Y \times X \times \R)$ is of the form $Q(y,x,s) = P(y,s)$ for some $P \in L^\infty(Y \times \R)$ satisfying
$$
P(g \cdot y, s+\log(\om(g,x))) = P(y,s) \quad\text{for all $g \in G$ and a.e.\ $(y,x,s) \in Y \times X \times \R$.}
$$
It follows that
$$P(y,s+\log(\om(g,x)) - \log(\om(g,x'))) = P(y,s)$$
for all $g \in G$ and a.e.\ $(y,x,x',s) \in Y \times X \times X \times \R$. For every $g \in G$, denote by $R_g$ the essential range of the map
$$X \times X \recht \R : (x,x') \mapsto \log(\om(g,x)) - \log(\om(g,x')) \; .$$
To conclude that $P \in L^\infty(Y) \ot 1$, it suffices to prove that $\bigcup_{g \in G} R_g$ generates a dense subgroup of $\R$. So it suffices to prove that there is no $a > 0$ such that $\log(\om(g,x)) - \log(\om(g,x')) \in \Z a$ for all $g \in G$ and a.e.\ $(x,x') \in X \times X$. Assume the contrary.

Fix $g,k \in G$ and define the measure preserving factor map
$$\pi : (\{0,1\} \times X,\mu_k \times \mu) \recht (X,\mu) : (\pi(z,x))_h = \begin{cases} x_h &\;\;\text{if $h \neq k$,}\\ z &\;\;\text{if $h = k$.}\end{cases}$$
By our assumption, $\log(\om(g,\pi(z,x))) - \log(\om(g,x)) \in \Z a$ for a.e.\ $z \in \{0,1\}$, $x \in X$. Since
$$\log(\om(g,x)) = \sum_{h \in G} \bigl( \log(\mu_{gh}(x_h)) - \log(\mu_h(x_h))\bigr)$$
with convergence a.e., we find that
$$\log(\om(g,\pi(z,x))) - \log(\om(g,x)) = \bigl(\log(\mu_{gk}(z)) - \log(\mu_k(z))\bigr) - \bigl(\log(\mu_{gk}(x_k)) - \log(\mu_k(x_k))\bigr)$$
for all $g \in G$ and a.e.\ $z \in \{0,1\}$, $x \in X$. Taking $z=1$ and $x_k = 0$, it follows that
$$\log \Bigl( \frac{\mu_{gk}(1)}{\mu_{gk}(0)} \Bigr) - \log \Bigl( \frac{\mu_{k}(1)}{\mu_{k}(0)} \Bigr) \in \Z a \; .$$
But the left hand side equals $\eps_1 (K(gk) - K(k))$. Since $g,k \in G$ were arbitrary and $K(e) = 0$, we conclude that $K(g) \in \Z (a/\eps_1)$ for all $g \in G$, contrary to our choice of $K$.

So, we have proven that $P \in L^\infty(Y) \ot 1$ and thus, $Q \in L^\infty(Y) \ot 1 \ot 1$. This means that $G \actson (X,\mu)$ is ergodic, of type III$_1$ and with weakly mixing Maharam extension.

By Lemma \ref{lem.lower-est-cocycle} below and \eqref{eq.choice-gap}, we get that
\begin{align*}
\sum_{g \in \cF} \int_X \sqrt{\om(g,x)} d\mu(x) &\geq \sum_{g \in \cF} \exp\Bigl(- \frac{3}{5} \|c_g\|_2^2 \Bigr) \geq \sum_{g \in \cF} \exp\Bigl(- \frac{3}{5} \eps_1^2 \|b_g\|_2^2 \Bigr) \\ &> (1-\eps_0) |\cF| \geq \Bigl\| \sum_{g \in \cF} \lambda_g \Bigr\| \; .
\end{align*}
So by Proposition \ref{prop.criterion-nonamenable} below, we conclude that the action $G \actson (X,\mu)$ is nonamenable.

It remains to prove that we may choose a $1$-cocycle $c : G \recht \ell^2(G)$ with $c_g = 0$ for all $g \in \Lambda$ and such that the associated function $\real H : G \recht \R$, determined by $H(e)=0$ and $c_g = H - g \cdot H$ for all $g \in G$, takes at least three different values. The space of $1$-cocycles $c : G \recht \ell^2(G)$ that vanish on $\Lambda$ is an $L(G)$-module of positive $L(G)$-dimension. It is in particular an infinite dimensional vector space. So we can choose $1$-cocycles $c,c' : G \recht \ell^2(G)$ that vanish on $\Lambda$ and such that the associated functions $\real H : G \recht \R$ and $\real H' : G \recht \R$ are $\R$-linearly independent and, in particular, nonzero. If either $\real H$ or $\real H'$ takes at least three values, we are done. Otherwise, after multiplying $c$ and $c'$ with nonzero real numbers, we may assume that $\real H = 1_A$ and $\real H' = 1_{A'}$, where $A,A'$ are distinct nonempty subsets of $G$. But then the function $\real H + 2 \real H'$, associated with the $1$-cocycle $c+2c'$, takes at least three different values.

Next assume that $\Lambda < G_0 < G$ are subgroups such that $\Lambda$ is infinite, $G_0 < G$ has finite index and $\beta_1^{(2)}(\Lambda) < \beta_1^{(2)}(G_0)$. Since $\beta_1^{(2)}(G_0) = [G:G_0] \, \beta_1^{(2)}(G)$, we also have that $\beta_1^{(2)}(G) > 0$. So if $\Lambda$ is amenable, we have $\beta_1^{(2)}(\Lambda) = 0 < \beta_1^{(2)}(G)$ and we can apply the first part of the proof. So we may assume that $\Lambda$ is nonamenable.

Choose a finite subset $\cF \subset G$ and $\eps_0 > 0$ such that \eqref{eq.choice-gap} holds. Since $\beta_1^{(2)}(\Lambda) < \beta_1^{(2)}(G_0)$, we can proceed as in the first part of the proof and find $\kappa_1,\kappa_2 > 0$ and a function $K : G_0 \recht [-\kappa_1,\kappa_2]$ satisfying the following properties.
\begin{itemlist}
\item The range of $K$ generates a dense subgroup of $\R$.
\item $K$ is invariant under left translation by $\Lambda$.
\item Writing $c_g(k) = K(k) - K(g^{-1}k)$ for all $g,k \in G_0$, we have that $c_g \in \ell^2(G_0)$ for all $g \in G_0$.
\end{itemlist}
Write $G = \sqcup_{i=1}^\kappa g_i G_0$. Define
$$F : G \recht [-\kappa_1,\kappa_2] : F(g_i h) = K(h) \quad\text{for all $i \in \{1,\ldots,\kappa\}$ and $h \in G_0$.}$$
For every $g,h \in G$, define $b_g(h) = F(h) - F(g^{-1} h)$. By construction, $b_g \in \ell^2(G)$ for every $g \in G$ and $G \recht \ell^2(G) : g \mapsto b_g$ is a cocycle. Note however that $b$ need not vanish on $\Lambda$.

For every $i \in \{1,\ldots,\kappa\}$, define the nonamenable group $\Lambda_i = g_i \Lambda g_i^{-1}$. By Schoenberg's theorem (see e.g.\ \cite[Theorem D.11]{BO08}), for every $\eps > 0$ and $i \in \{1,\ldots,\kappa\}$, the map
$$\vphi_{\eps,i} : \Lambda_i \recht \R  : h \mapsto \exp(-8 \, \eps^2 \, \|b_h\|_2^2)$$
is a positive definite function on $\Lambda_i$. When $\eps \recht 0$, we get that $\vphi_{\eps,i} \recht 1$ pointwise. Since $\Lambda_i$ is nonamenable, it follows that $\vphi_{\eps,i} \not\in \ell^2(\Lambda_i)$ for $\eps$ small enough. So we can choose $\eps_1 > 0$ small enough such that
\begin{align}
& \exp(\eps_1 \kappa_i) \leq 2 \quad\text{for $i=1,2$,} \quad \exp\Bigl(-\frac{3}{5} \, \eps_1^2 \, \|b_g\|_2^2\Bigr) > 1 -\eps_0 \quad\text{for all $g \in \cF$, and}\notag\\
& \sum_{h \in \Lambda_i} \exp\bigl(-16 \, \eps_1^2 \, \|b_h\|_2^2 \bigr) = +\infty \quad\text{for all $i \in \{1,\ldots,\kappa\}$.}\label{eq.makes-conservative}
\end{align}
For every $g \in G$, define the probability measure $\mu_g$ on $\{0,1\}$ given by
$$\mu_g(0) = \frac{1}{1+\exp(\eps_1 F(g))} \; .$$
Note that $\mu_g(0) \in [1/3,2/3]$ for all $g \in G$. Defining $d_g(h) = \mu_h(0) - \mu_{g^{-1}h}(0)$, we find that $\|d_g\|_2 \leq \eps_1 \, \|b_g\|_2$. So, $d_g \in \ell^2(G)$ and the Bernoulli action $G \actson (X,\mu) = \prod_{g \in G} (\{0,1\},\mu_g)$ is nonsingular and essentially free.

Choose an arbitrary pmp action $G \actson (Y,\eta)$ and consider the diagonal action $G \actson Y \times X \times \R$ of $G \actson Y$ and the Maharam extension $G \actson X \times \R$. Let $Q \in L^\infty(Y \times X \times \R)$ be $G$-invariant. We have to prove that $Q \in L^\infty(Y) \ot 1 \ot 1$.

For every subset $J \subset G$, define $(X_J,\mu_J) = \prod_{g \in J} (\{0,1\},\mu_g)$ and view $L^\infty(X_J) \subset L^\infty(X)$. Fix $i \in \{1,\ldots,\kappa\}$. We prove that $Q \in L^\infty(Y \times X_{G \setminus g_i G_0} \times \R)$. Since the map $K : G_0 \recht \R$ is $\Lambda$-invariant, we get that $\Lambda_i \actson (X_{g_i G_0},\mu_{g_i G_0})$ is a pmp Bernoulli action. By \eqref{eq.makes-conservative}, the inequality $\|d_g\|_2 \leq \eps_1 \, \|b_g\|_2$ and Proposition \ref{prop.conservative}, the action $\Lambda_i \actson X$ is conservative. This means that $\sum_{g \in \Lambda_i} \om(g,x) = +\infty$ for a.e.\ $x \in X$, so that also the diagonal action $\Lambda_i \actson Y \times X$ is conservative. A fortiori, the factor action $\Lambda_i \actson Y \times X_{G \setminus g_i G_0}$ is conservative and then also its Maharam extension $\Lambda_i \actson Y \times X_{G \setminus g_i G_0} \times \R$. Since we can view $\Lambda_i \actson Y \times X \times \R$ as the diagonal product of $\Lambda_i \actson Y \times X_{G \setminus g_i G_0} \times \R$ and the mixing pmp action $\Lambda_i \actson X_{g_i G_0}$, it follows from \cite[Theorem 2.3]{SW81} that the $\Lambda_i$-invariant functions in $L^\infty(Y \times X \times \R)$ belong to $L^\infty(Y \times X_{G \setminus g_i G_0} \times \R)$. So, $Q \in L^\infty(Y \times X_{G \setminus g_i G_0} \times \R)$.

Since this holds for every $i \in \{1,\ldots,\kappa\}$, it follows that $Q \in L^\infty(Y) \ovt 1 \ovt L^\infty(\R)$. We now proceed as in the first part of the proof. Since the range of $K$ generates a dense subgroup of $\R$, the same holds for $F$ and we conclude that $Q \in L^\infty(Y) \ot 1 \ot 1$.

The fact that $G \actson (X,\mu)$ is nonamenable in the sense of Zimmer follows exactly as in the first part of the proof.
\end{proof}

We now deduce Theorem \ref{thm.positive-betti-III1} from Lemma \ref{lem.construction} by proving that a group satisfying the assumptions of Theorem \ref{thm.positive-betti-III1} automatically admits subgroups $\Lambda < G_0 < G$ as in Lemma \ref{lem.construction}.

\begin{proof}[{Proof of Theorem \ref{thm.positive-betti-III1}}]
Let $G$ be a countable group with $\beta_1^{(2)}(G) > 0$, satisfying one of the properties in \ref{cond-G-one}--\ref{cond-G-four}. Since $\Z$ is amenable, case \ref{cond-G-one} follows from case \ref{cond-G-two}. In case \ref{cond-G-two}, if $\Lambda < G$ is an infinite amenable group, we have $\beta_1^{(2)}(\Lambda) = 0 < \beta_1^{(2)}(G)$ and taking $G_0 = G$, the assumptions of Lemma \ref{lem.construction} are satisfied.

Case \ref{cond-G-three} follows from case \ref{cond-G-four} by taking $G_0 = G$. So it remains to prove the theorem in case~\ref{cond-G-four}, i.e.\ in the presence of a finite index subgroup $G_0 < G$ with $[G:G_0] \geq \beta_1^{(2)}(G)^{-1}$. Then, $\beta_1^{(2)}(G_0) = [G:G_0] \, \beta_1^{(2)}(G) \geq 1$. Since we already proved the theorem in cases \ref{cond-G-one} and \ref{cond-G-two}, we may assume that $G_0$ is a torsion group without infinite amenable subgroups. We claim that there exist $a,b \in G_0$ such that the subgroup $\Lambda = \langle a,b\rangle$ generated by $a$ and $b$ is infinite. Indeed, if all two elements $a,b \in G_0$ generate a finite subgroup, it follows from \cite[Theorem 7]{St66} that $G_0$ contains an infinite abelian subgroup, contrary to our assumptions. So the claim is proved and we fix $a,b \in G_0$ generating an infinite subgroup $\Lambda = \langle a,b \rangle$.

We prove that $\beta_1^{(2)}(\Lambda) < 1$. Since $\beta_1^{(2)}(G_0) \geq 1$, the subgroups $\Lambda < G_0 < G$ then satisfy the assumptions of Lemma \ref{lem.construction}. Assume that $a$ has order $n$ and $b$ has order $m$. Since any cocycle $\gamma : \Lambda \recht \ell^2(\Lambda)$ is cohomologous to a cocycle that vanishes on the finite subgroup generated by $a$ and is then entirely determined by its value on $b$, we find that
\begin{align*}
\beta_1^{(2)}(\Lambda) = \dim_{L(\Lambda)}\bigl( \{\xi \in \ell^2(\Lambda) \mid \; & \text{there exists a $1$-cocycle $\gamma: \Lambda \recht \ell^2(\Lambda)$ with}\\
& \text{$\gamma_a = 0$ and $\gamma_b = \xi$} \; \} \bigr) \\ &\quad\quad\quad - \dim_{L(\Lambda)}\bigl( \{\eta - b \cdot \eta \mid \eta \in \ell^2(\Lambda) \;\; , \;\; a \cdot \eta = \eta \}\bigr) \; .
\end{align*}
The first term is bounded by $1$. Because $\Lambda$ is infinite and $a$ has order $n$, the second term equals
$$\dim_{L(\Lambda)} \bigl(\{ \eta \in \ell^2(\Lambda) \mid a \cdot \eta = \eta\}\bigr) = \frac{1}{n} \; .$$
So, $\beta_1^{(2)}(\Lambda) \leq 1- 1/n < 1$. This concludes the proof of the theorem.
\end{proof}

The following result is implicitly contained in the proof of \cite[Theorem 7]{DN10}. For completeness, we provide a detailed proof.

\begin{proposition}[{\cite[Theorem 7]{DN10}}]\label{prop.criterion-nonamenable}
Let $G$ be a countable group and $G \actson (X,\mu)$ a nonsingular action. Denote by $\om : G \times X \recht (0,+\infty)$ the Radon-Nikodym cocycle given by \eqref{eq.general-RN}. Denote by $\lambda : G \recht \cU(\ell^2(G))$ the left regular representation. If there exists a finite subset $\cF \subset G$ such that
$$
\sum_{g \in \cF} \int_X \sqrt{\om(g,x)} d\mu(x) > \Bigl\| \sum_{g \in \cF} \lambda_g \Bigr\| \; ,
$$
then the action $G \actson (X,\mu)$ is nonamenable in the sense of Zimmer.
\end{proposition}
\begin{proof}
Assume that $G \actson (X,\mu)$ is amenable in the sense of Zimmer and fix a finite subset $\cF \subset G$. Since $G \actson (X,\mu)$ is amenable, we can take a sequence $\xi_n \in L^\infty(X,\ell^2(G))$ such that $\|\xi_n(x)\|_2 = 1$ for a.e.\ $x \in X$ and
$$\lim_n \int_X \langle \lambda_g \xi_n(g^{-1} \cdot x) , \xi_n(x)\rangle \, H(x) \, d\mu(x) = \int_X  H(x) \, d\mu(x) \quad\text{for every $H \in L^1(X,\mu)$ and $g \in G$.}$$
Define the Hilbert space $\cK = L^2(X,\ell^2(G))$ and the unitary representation
$$\pi : G \recht \cU(\cK) : (\pi(g)\xi)(x) = \sqrt{\om(g^{-1},x)} \, \lambda_g \xi(g^{-1} \cdot x) \; .$$
We view $\xi_n$ as a sequence of unit vectors in $\cK$ and find that
$$\lim_n \langle \pi(g) \xi_n , \xi_n \rangle = \lim_n \langle \xi_n, \pi(g^{-1}) \xi_n \rangle = \int_X \sqrt{\om(g,x)} \, d\mu(x) \; .$$
It follows that
$$\sum_{g \in \cF} \int_X \sqrt{\om(g,x)} \, d\mu(x) \leq \Bigl\| \sum_{g \in \cF} \pi(g) \Bigr\| \; .$$
Defining the closed subspace $\cK_0 \subset \cK$ given by $\cK_0 = L^2(X,\C \delta_e)$, we see that the subspaces $\pi(g) \cK_0$, $g \in G$, are mutually orthogonal and that these subspaces densely span $\cK$. Therefore, $\pi$ is unitarily equivalent with a multiple of the regular representation of $G$. Therefore,
$$\Bigl\| \sum_{g \in \cF} \pi(g) \Bigr\| = \Bigl\| \sum_{g \in \cF} \lambda_g \Bigr\|$$
and the proposition is proved.
\end{proof}

\begin{lemma}\label{lem.lower-est-cocycle}
Let $G \actson I$ be a free action of the countable group $G$ on the countable set $I$ and let $F : I \recht (0,1)$ be a function satisfying \eqref{eq.conds-F} with $\delta = 1/3$. Denote by $G \actson (X,\mu)$ the associated Bernoulli action, by $\om : G \times X \recht (0,+\infty)$ its Radon-Nikodym cocycle and by $c : G \recht \ell^2(I)$ the associated $1$-cocycle as in \eqref{eq.one-cocycle-c}. Then,
\begin{equation}\label{eq.lower-est}
\int_X \sqrt{\om(g,x)} \, d\mu(x) \geq \exp\Bigl(- \frac{3}{5} \|c_g\|_2^2 \Bigr) \quad\text{for all $g \in G$.}
\end{equation}
\end{lemma}
\begin{proof}
Let $I = \{i_1,i_2,\ldots\}$ be an enumeration of $I$. Define
$$\om_n : G \times X \recht (0,+\infty) : \om_n(g,x) = \prod_{k=1}^n \frac{\mu_{g\cdot i_k}(x_{i_k})}{\mu_{i_k}(x_{i_k})} \; .$$
Fix $g \in G$. By \cite{Ka48}, we know that $\om_n(g,x) \recht \om(g,x)$ for a.e.\ $x \in X$ and that $\sqrt{\om_n(g,\cdot)} \recht \sqrt{\om(g,\cdot)}$ in $L^2(X,\mu)$. Therefore,
\begin{equation}\label{eq.stapje}
\int_X \sqrt{\om(g,x)} \, d\mu(x) = \lim_n \prod_{k=1}^n \bigl( \sqrt{F(i_k) F(g \cdot i_k)} + \sqrt{(1-F(i_k))(1-F(g \cdot i_k))} \bigr) \; .
\end{equation}
For all $1/3 \leq a,b \leq 2/3$, we have that
$$\sqrt{ab} + \sqrt{(1-a)(1-b)} \geq 1 - \frac{9}{16}(b-a)^2 \; .$$
For every $0 \leq t \leq 1/16$, we have that $\log(1-t) \geq - (16/15)t$. Since $\frac{9}{16}(b-a)^2$ lies between $0$ and $1/16$, we get that
$$\log(\sqrt{ab} + \sqrt{(1-a)(1-b)}) \geq -\frac{3}{5}(b-a)^2 \; .$$
It then follows from \eqref{eq.stapje} that
$$\int_X \sqrt{\om(g,x)} \, d\mu(x) \geq \exp\Bigl(- \frac{3}{5} \sum_{i \in I} (F(i) - F(g \cdot i))^2 \Bigr) = \exp\bigl(- \frac{3}{5} \|c_g\|_2^2 \bigr) \; .$$
So \eqref{eq.lower-est} holds and the lemma is proved.
\end{proof}

\section{Amenable groups}

\begin{theorem}\label{thm.amenable-III1}
Let $G$ be an amenable countable infinite group. Then $G$ admits a nonsingular Bernoulli action $G \actson (X,\mu) = \prod_{g \in G} (X_0,\mu_g)$ that is essentially free, ergodic and of type III$_1$ and that has a weakly mixing Maharam extension.

In the following two cases, we can choose as a base space $X_0$ the two point set $\{0,1\}$:
\begin{itemlist}
\item when $G$ has at least one element of infinite order;
\item when $G$ admits an infinite subgroup of infinite index.
\end{itemlist}
\end{theorem}

The only amenable groups $G$ that do not satisfy any of the extra assumptions in Theorem \ref{thm.amenable-III1} are the amenable torsion groups with the property that every subgroup is either finite or of finite index. While it is unknown whether there are finitely generated such groups, the locally finite Pr\"{u}fer $p$-groups, for $p$ prime, given as the direct limit of the finite groups $\Z / p^n \Z$, have the property that every proper subgroup is finite. We do not know whether these groups admit a nonsingular Bernoulli action of type III with base space $X_0 = \{0,1\}$.

In \cite[Theorem 7]{Ko10}, it is proven that there exist nonsingular Bernoulli shifts $T$ that are ergodic, of type III$_1$ and \emph{power weakly mixing} in the sense that all transformations $T^{a_1} \times \cdots \times T^{a_k}$ remain ergodic. Our proof of Theorem \ref{thm.amenable-III1} also gives the following concrete examples.

\begin{corollary}\label{cor.explicit-Z}
Let $0 < \lambda < 1$ and put $n_0 = \lceil (1-\lambda)^{-2} \rceil$. Define for every $n \in \Z$, the probability measure $\mu_n$ on $\{0,1\}$ given by
$$\mu_n(0) = \begin{cases} \lambda + \frac{1}{\sqrt{n \log(n)}} &\;\;\text{if $n \geq n_0$,}\\
\lambda &\;\;\text{if $n < n_0$.}\end{cases}$$
The associated Bernoulli shift $T$ on $(X,\mu) = \prod_{n \in \Z} (\{0,1\},\mu_n)$ is essentially free, ergodic, of type III$_1$ and with weakly mixing Maharam extension. Moreover, for all $k \geq 1$ and $a_1,\ldots,a_k \in \Z \setminus \{0\}$, the nonsingular transformation
$$T^{a_1} \times \cdots \times T^{a_k} : X^k \recht X^k : (x_1,\ldots,x_k) \mapsto (T^{a_1}(x_1),\ldots,T^{a_k}(x_k))$$
remains ergodic, of type III$_1$ and with weakly mixing Maharam extension.
\end{corollary}

As another application of our methods, we give the following concrete example of an ergodic type III$_1$ Bernoulli shift that is not power weakly mixing. As far as we know, such examples were not given before.

\begin{corollary}\label{cor.explicit-Z-power-dissipative}
Define for every $n \in \Z$, the probability measure $\mu_n$ on $\{0,1\}$ given by
$$\mu_n(0) = \begin{cases} \frac{1}{2} + \frac{1}{6 \sqrt{n}} &\;\;\text{if $n \geq 1$,}\\
\frac{1}{2} &\;\;\text{if $n \leq 0$.}\end{cases}$$
The associated Bernoulli shift $T$ on $(X,\mu) = \prod_{n \in \Z} (\{0,1\},\mu_n)$ is essentially free, ergodic, of type III$_1$ and with weakly mixing Maharam extension, but for $m$ large enough (e.g.\ $m \geq 73$), the $m$-th power transformation
$$T \times \cdots \times T : X^m \recht X^m : (x_1,\ldots,x_m) \mapsto (T(x_1),\ldots,T(x_m))$$
is dissipative.
\end{corollary}

Theorem \ref{thm.amenable-III1} and its corollaries are proved in Sections \ref{sec.proof-amenable-III1}--\ref{sec.proof-last-cor}.

\subsection{\boldmath Determining the type: removing inessential subsets of $I$}\label{sec.reduction}

Fix a countable infinite group $\Lambda$ acting freely on a countable set $I$ and fix a function $F : I \recht (0,1)$ satisfying \eqref{eq.conds-F}. Define the probability measures $\mu_i$ on $\{0,1\}$ given by $\mu_i(0) = F(i)$. Denote by $\Lambda \actson (X,\mu) = \prod_{i \in I} (\{0,1\},\mu_i)$ the associated Bernoulli action with Radon-Nikodym cocycle $\om : \Lambda \times X \recht (0,+\infty)$ given by \eqref{eq.RN-conv-ae} and Maharam extension $\Lambda \actson (X \times \R,\mu \times \nu)$ given by \eqref{eq.maharam}. Fix an arbitrary pmp action $\Lambda \actson (Y,\eta)$ and consider the diagonal action $\Lambda \actson (Y \times X \times \R, \eta \times \mu \times \nu)$.

For every subset $J \subset I$, we consider $(X_J,\mu_J) = \prod_{j \in J} (\{0,1\},\mu_j)$. We denote by $x \mapsto x_J$ the natural measure preserving factor map $(X,\mu) \recht (X_J,\mu_J)$. Given $0 < \lambda < 1$, we denote by $\nu_\lambda$ the probability measure on $\{0,1\}$ given by $\nu_\lambda(0) = \lambda$. We also use the notation
\begin{equation}\label{eq.maps-vphi-i}
\vphi_{\lambda,i} : \{0,1\} \recht \R : \vphi_{\lambda,i}(x) = \log \frac{\mu_i(x)}{\nu_\lambda(x)} = \begin{cases} \log(F(i)) - \log(\lambda) &\;\;\text{if $x = 0$,} \\ \log(1-F(i)) - \log(1-\lambda) &\;\;\text{if $x=1$.}\end{cases}
\end{equation}

We introduce the following ad hoc terminology: given $0 < \lambda < 1$, we call a subset $J \subset I$ \emph{$\lambda$-inessential} if the following two conditions hold.
\begin{enumlist}
\item $\mu_j = \nu_\lambda$ for all but finitely many $j \in J$.
\item For every $\Lambda$-invariant $Q \in L^\infty(Y \times X \times \R)$, there exists a $P \in L^\infty(Y \times X_{I \setminus J} \times \R)$ such that $\|P\|_\infty \leq \|Q\|_\infty$ and
$$Q(y,x,s) = P\Bigl(y,x_{I \setminus J} \; , \; s - \sum_{j \in J} \vphi_{\lambda,j}(x_j) \Bigr) \quad\text{for a.e.\ $(y,x,s) \in Y \times X \times \R$} \; .$$
\end{enumlist}
Note that the sum over $j \in J$ is actually a finite sum since $\vphi_{\lambda,j}$ is the zero map for all but finitely many $j \in J$. The terminology ``inessential'' is motivated by the fact that these subsets ``do not contribute'' to the type of the action $\Lambda \actson (X,\mu)$.

Note that if we assume that condition~1 holds, then condition~2 is equivalent with the following: denoting by $\mu' \sim \mu$ the measure given by $\mu'_i = \mu_i$ for all $i \in I \setminus J$ and $\mu'_j = \nu_\lambda$ for all $j \in J$, every $\Lambda$-invariant function in $L^\infty(Y \times X \times \R)$ w.r.t.\ the diagonal action of $\Lambda \actson (Y,\eta)$ and the Maharam extension for $\Lambda \actson (X,\mu')$ belongs to $L^\infty(Y \times X_{I \setminus J} \times \R)$. Using this characterization, it follows that the union of two inessential subsets is again inessential.

We provide two criteria for subsets $J \subset I$ to be inessential.

\begin{proposition}\label{prop.inessential-orbit}
Assume that $\Lambda \actson (X,\mu)$ is conservative (see Section \ref{sec.conservative}). Let $0 < \lambda < 1$. If $i_0 \in I$ is such that $F(g \cdot i_0) = \lambda$ for all but finitely many $g \in \Lambda$, then $\Lambda \cdot i_0 \subset I$ is $\lambda$-inessential.
\end{proposition}
\begin{proof}
Write $J = \Lambda \cdot i_0$ and replace $\mu$ by the equivalent measure satisfying $\mu_j = \nu_{\lambda}$ for all $j \in J$. We have to prove that every $\Lambda$-invariant function $Q \in L^\infty(Y \times X \times \R)$ for the diagonal product of the fixed pmp action $\Lambda \actson (Y,\eta)$ and the Maharam extension $\Lambda \actson X \times \R$ belongs to $L^\infty(Y \times X_{I \setminus J} \times \R)$.

Since a nonsingular action $\Lambda \actson (X,\mu)$ is conservative if and only if $\sum_{g \in \Lambda} \om(g,x) = +\infty$ for a.e.\ $x \in X$, it follows that also the diagonal action $\Lambda \actson Y \times X$ is conservative. Note that $\Lambda \actson (X_J,\mu_J)$ is a probability measure preserving Bernoulli action and that $\Lambda \actson Y \times X \times \R$ can be viewed as the product of the action $\Lambda \actson Y \times X_{I \setminus J} \times \R$ and the action $\Lambda \actson X_J$. The action $\Lambda \actson (Y \times X_{I \setminus J},\eta \times \mu_{I \setminus J})$ is a factor of the action $\Lambda \actson (Y \times X,\eta \times \mu)$. Since the inverse image of a wandering set under a factor map remains wandering, it follows that $\Lambda \actson (Y \times X_{I \setminus J},\eta \times \mu_{I \setminus J})$ is conservative. Then also its Maharam extension $\Lambda \actson Y \times X_{I \setminus J} \times \R$ is conservative. Since the probability measure preserving Bernoulli action $\Lambda \actson X_J$ is mixing, it follows from \cite[Theorem 2.3]{SW81} that the $\Lambda$-invariant functions in $L^\infty(Y \times X \times \R)$ belong to $L^\infty(Y \times X_{I \setminus J} \times \R)$.
\end{proof}

Our next criterion for being inessential is a consequence of \cite[Lemma 4.3]{ST94}, saying the following. Assume that
\begin{itemlist}
\item $(Z,\zeta)$ and $(Z_0,\zeta_0)$ are $\sigma$-finite standard measure spaces, with $\sigma$-algebras of measurable sets $\cB$ and $\cB_0$;
\item $\pi : Z \recht Z_0$ is a measure preserving factor map;
\item $T : Z \recht Z$ is a measure preserving, conservative automorphism and $T_0 : Z_0 \recht Z_0$ is a measure preserving endomorphism;
\item $\pi \circ T = T_0 \circ \pi$ a.e.;
\item $\cB$ is, up to measure zero, generated by $\{T^k(\pi^{-1}(\cB_0)) \mid k \in \Z\}$.
\end{itemlist}
Then by \cite[Lemma 4.3]{ST94}, every $T$-invariant function $Q \in L^\infty(Z)$ factors through $\pi$.

\begin{proposition}\label{prop.inessential-Z}
Assume that $\Lambda = \Z$ and let $0 < \lambda < 1$. Assume that $\Z \actson (X,\mu)$ is conservative. If $i_0 \in I$ is such that $F(n \cdot i_0) = \lambda$ for all $n \geq 0$, then $\{n \cdot i_0 \mid n \geq n_0\}$ is $\lambda$-inessential for every $n_0 \in \Z$.

Similarly, if $i_0 \in I$ such that $F(n \cdot i_0) = \lambda$ for all $n \leq 0$, then $\{n \cdot i_0 \mid n \leq n_0\}$ is $\lambda$-inessential for every $n_0 \in \Z$.
\end{proposition}
\begin{proof}
By symmetry, it suffices to prove the first statement. Fix $n_0 \in \Z$. Replace $i_0$ by $n_0 \cdot i_0$ and replace $\mu$ by the equivalent measure satisfying $\mu_j = \nu_\lambda$ for all $j \in J := \{n \cdot i_0 \mid n \geq 0\}$. Write $J' = I \setminus J$. We have to prove that every $\Z$-invariant function $Q \in L^\infty(Y \times X \times \R)$ for the diagonal product of the fixed pmp action $\Z \actson (Y,\eta)$ and the Maharam extension $\Z \actson X \times \R$ belongs to $L^\infty(Y \times X_{J'} \times \R)$.

Denote by
$$T : Y \times X \times \R \recht Y \times X \times \R : T(y,x,s) = 1 \cdot (y,x,s) = (1 \cdot y, 1 \cdot x, \log(\om(1,x)) + s)$$
the $\eta \times \mu \times \nu$-preserving transformation given by $1 \in \Z$. Define the $\eta \times \mu_{J'} \times \nu$-preserving endomorphism
\begin{align*}
T_0 : Y \times X_{J'} \times \R \recht Y \times X_{J'} \times \R : T_0(y,x,s) = & \; (1 \cdot y, x',\log(\om(1,x))+s) \\ & \; \text{where}\;\; x'_i = x_{(-1)\cdot i} \;\;\text{for all}\;\; i \in J' \; ,
\end{align*}
which is well defined because $x \mapsto \om(1,x)$ factors through $X_{J'}$ by \eqref{eq.RN-conv-ae}.

As in the proof of Proposition \ref{prop.inessential-orbit}, since $\Z \actson (X,\mu)$ is conservative, also the diagonal action $\Z \actson (Y \times X,\eta \times \mu)$ is conservative. So, the Maharam extension $T$ is conservative as well. Applying \cite[Lemma 4.3]{ST94} as in the discussion before the proposition to the natural, measure preserving factor map $Y \times X \times \R \recht Y \times X_{J'} \times \R$, it follows that the $\Lambda$-invariant functions in $L^\infty(Y \times X \times \R)$ belong to $L^\infty(Y \times X_{J'} \times \R)$.
\end{proof}

\subsection{Determining the type: reduction to the tail}

Fix a countable infinite group $\Lambda$ acting freely on a countable set $I$ and fix a function $F : I \recht (0,1)$ satisfying \eqref{eq.conds-F}. Denote by $\Lambda \actson (X,\mu)$ the associated Bernoulli action.

\begin{proposition}\label{prop.inessential-type-III1}
Assume that $0 < \lambda < 1$ such that
$$\lim_{i \recht \infty} F(i) = \lambda \quad\text{and}\quad \sum_{i \in I} (F(i)-\lambda)^2=+\infty \; .$$
Assume that there exists a sequence of $\lambda$-inessential subsets $J_n \subset I$ (see Section \ref{sec.reduction}) such that $\bigcup_n J_n = I$. Then, $\Lambda \actson (X,\mu)$ is ergodic and of type III$_1$, and has a weakly mixing Maharam extension.
\end{proposition}

\begin{proof}
Enumerate $I = \{i_1,i_2,\ldots\}$. Define $I_n = \{i_1,\ldots,i_n\}$. For every $n \geq 1$, there exists an $m \geq 1$ such that $I_n \subset \bigcup_{k=1}^m J_k$. Since the union of two $\lambda$-inessential subsets is inessential and since subsets of $\lambda$-inessential sets are again $\lambda$-inessential, it follows that $I_n$ is $\lambda$-inessential for every $n$. Write $I_n' = I \setminus I_n$.

Fix a pmp action $\Lambda \actson (Y,\eta)$. We have to prove that every $\Lambda$-invariant element $Q \in L^\infty(Y \times X \times \R)$ for the diagonal product of $\Lambda \actson (Y,\eta)$ and the Maharam extension $\Lambda \actson (X \times \R, \mu \times \nu)$ given by \eqref{eq.maharam} belongs to $L^\infty(Y) \ot 1 \ot 1$. Using the notation in \eqref{eq.maps-vphi-i}, we find $Q_n \in L^\infty(Y \times X_{I_n'} \times \R)$ with $\|Q_n\|_\infty \leq \|Q\|_\infty$ for all $n$ and
\begin{equation}\label{eq.stap}
Q(y,x,s) = Q_n\Bigl(y, x_{I_n'} \; , \; s - \sum_{j \in I_n} \vphi_{\lambda,j}(x_j) \Bigr) \quad\text{for a.e.\ $(x,s) \in X \times \R$} \; .
\end{equation}
Define $S_n \in L^\infty(Y \times X_{I_n} \times \R)$ as the conditional expectation of $Q$ onto $L^\infty(Y \times X_{I_n} \times \R)$. By martingale convergence, we have that $S_n(y,x,s) \recht Q(y,x,s)$ for a.e.\ $(x,s) \in X \times \R$. Define $P_n \in L^\infty(Y \times \R)$ such that $(P_n)_{13}$ is the conditional expectation of $Q_n$ onto $L^\infty(Y) \ovt 1 \ovt L^\infty(\R)$. Then $\|P_n\|_\infty \leq \|Q_n\|_\infty \leq \|Q\|_\infty$ and it follows from \eqref{eq.stap} that
\begin{equation}\label{eq.form-Hn}
S_n(y,x,s) = P_n\Bigl(y,s - \sum_{i \in I_n} \vphi_{\lambda,i}(x_i) \Bigr) \; .
\end{equation}
Denote by $\cR$ the \emph{tail equivalence relation} on $(X,\mu)$ given by $(x,x') \in \cR$ if and only if $x_i = x'_i$ for all but finitely many $i \in I$. Define the $1$-cocycle
$$\al : \cR \recht \R : \al(x,x') = \sum_{i \in I} (\vphi_{\lambda,i}(x_i) - \vphi_{\lambda,i}(x'_i)) \; .$$
Denote by $\cR(\al)$ the associated skew product, i.e.\ the equivalence relation on $X \times \R$ given by $(x,s) \sim (x',t)$ if and only if $(x,x') \in \cR$ and $s = \al(x,x') + t$. Denote by $\cS(\al)$ the equivalence relation on $Y \times X \times \R$ given by $\id \times \cR(\al)$, i.e.\ with $(y,x,s) \sim_{\cS(\al)} (y',x',t)$ if and only if $y=y'$ and $(x,s) \sim_{\cR(\al)} (x',t)$.

We claim that $Q \in L^\infty(Y \times X \times \R)$ is $\cS(\al)$-invariant. Define $\si : \{0,1\} \recht \{0,1\}$ given by $\si(0)=1$ and $\si(1) = 0$. For every $i \in I$, define
\begin{align*}
\si_i : Y \times X \times \R \recht Y \times X \times \R : \si_i(y,x,s) = & \; (y,x',s-\vphi_{\lambda,i}(x_i) + \vphi_{\lambda,i}(\si(x_i))) \\ & \; \text{where $x'_j = x_j$ if $j \neq i$ and $x'_i = \si(x_i)$.}
\end{align*}
Since the graphs of the automorphisms $(\si_i)_{i \in I}$ generate the equivalence relation $\cS(\al)$, to prove the claim, it suffices to prove that $Q(\si_i(y,x,s)) = Q(y,x,s)$ for all $i \in I$ and a.e.\ $(y,x,s) \in Y \times X \times \R$. Whenever $i \in I_n$, it follows from \eqref{eq.form-Hn} that $S_n(\si_i(y,x,s)) = S_n(y,x,s)$ for all $(y,x,s)$. Since $S_n \recht Q$ a.e.\ and $i \in I_n$ for $n$ large enough, the claim is proven.

By \cite[Proposition 1.5]{DL16}, the cocycle $\al$ is ergodic, meaning that $\cR(\al)$ is an ergodic equivalence relation. So every $\cS(\al)$-invariant element $Q \in L^\infty(Y \times X \times \R)$ belongs to $L^\infty(Y) \ot 1 \ot 1$. Therefore, $Q \in L^\infty(Y) \ot 1 \ot 1$ and the proposition is proven.
\end{proof}

\subsection{\boldmath Bernoulli actions of the group $\Z$}

Combining Propositions \ref{prop.conservative}, \ref{prop.inessential-Z} and \ref{prop.inessential-type-III1}, we get the following result that we use to construct numerous concrete examples of type III$_1$ Bernoulli actions of $\Z$.

\begin{proposition}\label{prop.Bernoulli-Z-generic}
Let $I$ be a countable set and $\Z \actson I$ a free action. Let $0 < \delta < 1$ and $\kappa > \delta^{-2} + \delta^{-1}(1-\delta)^{-2}$. Assume that $F : I \recht [\delta,1-\delta]$ is a function satisfying the following conditions.
\begin{enumlist}
\item\label{as.one} There exists a $0 < \lambda < 1$ such that $\lim_{i \recht \infty} F(i) = \lambda$ and $\sum_{i \in I} (F(i)-\lambda)^2 = +\infty$.
\item\label{as.two} For every $k \in \Z$, the function $c_k : I \recht \R : c_k(i) = F(i) - F((-k)\cdot i)$ belongs to $\ell^2(I)$.
\item\label{as.three} We have $\sum_{k \in \Z} \exp(-\kappa \|c_k\|_2^2) = +\infty$.
\item\label{as.four} For every $i \in I$, there exist $n_i \in \Z$ and $\eps_i \in \{1,-1\}$ such that $F(n\cdot i) = \lambda$ for all $n \in \Z$ with $\eps_i n \leq n_i$.
\end{enumlist}
Then, the Bernoulli action $\Z \actson (X,\mu) = \prod_{i \in I} (\{0,1\},\mu_i)$ with $\mu_i(0) = F(i)$ is nonsingular, essentially free, ergodic and of type III$_1$, and has a weakly mixing Maharam extension.
\end{proposition}
\begin{proof}
By \ref{as.two}, the Bernoulli action $\Z \actson (X,\mu)$ is nonsingular. Since $\delta \leq \mu_i(0) \leq 1-\delta$ for all $i \in I$, the action is essentially free. By \ref{as.three} and Proposition \ref{prop.conservative}, the action $\Z \actson (X,\mu)$ is conservative. By \ref{as.four} and Proposition \ref{prop.inessential-Z}, the subset $\{n \cdot i \mid n \in \Z , \eps_i n \leq m\} \subset I$ is $\lambda$-inessential for every $i \in I$ and every $m \in \Z$. Since these subsets cover $I$, it follows from \ref{as.one} and Proposition \ref{prop.inessential-type-III1} that $\Z \actson (X,\mu)$ is ergodic and of type III$_1$, and that its Maharam extension is weakly mixing.
\end{proof}

\subsection{\boldmath Amenable groups have $1$-cocycles of arbitrarily small growth}\label{sec.cocycles-slow-growth}

A countable group $G$ has the \emph{Haagerup property} if there exists a proper $1$-cocycle $c : G \recht \cH$ into \emph{some} unitary representation $\pi : G \recht \cU(\cH)$. In \cite[Proposition 3.10]{CTV05}, it is proven that a group with the Haagerup property admits such proper $1$-cocycles $c : G \recht \cH$ of arbitrary slow growth. In \cite{BCV93}, it is proven that all amenable groups have the Haagerup property. Mimicking that proof, we show that an amenable group $G$ admits a proper $1$-cocycle $c : G \recht \ell^2(G)$ of arbitrary slow growth.

A function $\vphi : I \recht [0,+\infty)$ on a countable infinite set $I$ is called \emph{proper} if $\{i \in I \mid \vphi(i) \leq \kappa\}$ is finite for every $\kappa > 0$.

Recall that a F{\o}lner sequence for an amenable group $G$ is a sequence of finite, nonempty subsets $A_n \subset G$ satisfying
$$\lim_n \frac{|g A_n \triangle A_n|}{|A_n|} = 0 \quad\text{for all}\;\; g \in G \; .$$

\begin{proposition}\label{prop.cocycle-small-growth}
Let $G$ be an amenable countable infinite group and $\vphi : G \recht [0,+\infty)$ a proper function with $\vphi(g) > 0$ for all $g \neq e$. Then there exists a $1$-cocycle $c : G \recht \ell^2(G)$ such that $\|c_g\|_2 \leq \vphi(g)$ for every $g \in G$ and such that $g \mapsto \|c_g\|_2$ is proper.

More concretely, given $\vphi$, given any F{\o}lner sequence $A_n \subset G$ with all $A_n$ being disjoint and given $\delta > 0$, we can pass to a subsequence and choose $\eps_n \in (0,\delta)$ such that
\begin{itemlist}
\item $\lim_n \eps_n = 0$ and $\sum_n \eps_n^2 = +\infty$,
\item the function
\begin{equation}\label{eq.def-F}
F : G \recht [0,\delta) : F(g) = \begin{cases} \eps_n / \sqrt{|A_n|} &\;\;\text{if $g \in A_n$ for some $n$,} \\ 0 &\;\;\text{if $g \not\in \bigcup_n A_n$,}\end{cases}
\end{equation}
is such that $c_g(k) = F(k) - F(g^{-1}k)$ defines a $1$-cocycle $c : G \recht \ell^2(G)$ with the properties that $\|c_g\|_2 \leq \vphi(g)$ for every $g \in G$ and that $g \mapsto \|c_g\|_2$ is proper.
\end{itemlist}
\end{proposition}

\begin{proof}
Enumerate $G = \{g_0,g_1,g_2,\ldots\}$ with $g_0 = e$. Choose a sequence $\eps_n \in (0,\delta)$ such that $\lim_n \eps_n = 0$, $\sum_n \eps_n^2 = +\infty$ and
$$\sum_{n=1}^k \eps_n^2 \leq \frac{1}{2} \vphi(g_k)^2 \quad\text{for all}\;\; k \geq 1 \; .$$
After passing to a subsequence of $A_n$, we may assume that
$$\eps_n^2 \, \frac{|g_k A_n \triangle A_n|}{|A_n|} \leq \eps_k^2 \, 2^{-n} \quad\text{for all}\;\; 1 \leq k \leq n \; .$$
Define the function $F$ as in \eqref{eq.def-F}. For every $k \geq 1$, we have
\begin{align*}
\|g_k \cdot F^2 - F^2 \|_1 & \leq 2 \sum_{n=1}^{k-1} \eps_n^2 + \sum_{n=k}^\infty \eps_n^2 \, \frac{|g_k A_n \triangle A_n|}{|A_n|}
\leq 2 \sum_{n=1}^{k-1} \eps_n^2 + \sum_{n=k}^\infty \eps_k^2 \, 2^{-n} \\
& \leq 2 \sum_{n=1}^{k-1} \eps_n^2 + 2 \eps_k^2 = 2 \sum_{n=1}^k \eps_k^2 \leq \vphi(g_k)^2 \; .
\end{align*}
Since $\|g_k \cdot F - F\|_2^2 \leq \|g_k \cdot F^2 - F^2 \|_1 \leq \vphi(g_k)^2$, we indeed find that the $1$-cocycle $c : G \recht \ell^2(G)$ defined by $c_g(k) = F(k) - F(g^{-1}k)$ satisfies $\|c_g\|_2 \leq \vphi(g)$ for all $g \in G$.

Since $\lim_{g \recht \infty} F(g) = 0$ and $\sum_g F(g)^2 = +\infty$, the $1$-cocycle $c$ is not inner. By \cite[Theorem 2.5]{PT10}, the $1$-cocycle $c$ is proper.
\end{proof}

\subsection{Proof of Theorem \ref{thm.amenable-III1}}\label{sec.proof-amenable-III1}

We first prove that in the following two cases, the group $G$ admits a nonsingular Bernoulli action $G \actson \prod_{g \in G} (\{0,1\},\mu_g)$ with base space $\{0,1\}$ satisfying the conclusions of Theorem \ref{thm.amenable-III1}.

{\bf Case 1.} $G$ is an amenable group that admits an infinite subgroup $\Lambda$ of infinite index.

{\bf Case 2.} $G$ admits a copy of $\Z$ as a finite index subgroup.

{\bf Proof in case 1.} We start by proving that $G$ admits a F{\o}lner sequence $A_n \subset G$ for which all the sets $\Lambda A_n$ are disjoint. To prove this claim, let $B_n \subset G$ be an arbitrary
F{\o}lner sequence and define the left invariant mean $m$ on $G$ as a limit point of the means $m_n(C) = |C \cap B_n| / |B_n|$. Since $\Lambda < G$ has infinite index, we can fix a sequence $g_n \in G$ such that the sets $g_n \Lambda$ are disjoint. It follows that for every fixed $h \in G$, the sets $g_n \Lambda h$ are disjoint. By left invariance, this forces $m(\Lambda h) = 0$. So, for every finite subset $\cF \subset G$, we get that $m(\Lambda \cF) = 0$. This implies that after passing to a subsequence of $B_n$, we may assume that $|\Lambda \cF \cap B_n| / |B_n| \recht 0$ for every finite subset $\cF \subset G$.

Write $G$ as an increasing union of finite subsets $\cF_n \subset G$ and choose $\cF_n$ such that $B_k \subset \cF_n$ for all $k < n$. Choose inductively $s_1 < s_2 < \cdots$ such that
$$\frac{|\Lambda \cF_{s_{n-1}} \cap B_{s_n}|}{|B_{s_n}|} < \frac{1}{n}$$
for all $n \geq 1$. Defining $A_n = B_{s_n} \setminus \Lambda \cF_{s_{n-1}}$, we have found a F{\o}lner sequence $A_n \subset G$ for which all the sets $\Lambda A_n$ are disjoint.

Let $G = \{g_0,g_1,g_2,\ldots\}$ be an enumeration of the group $G$ such that $g_0 = e$ and $\{g_0,g_2,g_4,\ldots\}$ is an enumeration of the infinite subgroup $\Lambda$. By Proposition \ref{prop.cocycle-small-growth}, we can pass to a subsequence of $A_n$ and choose $\eps_n \in (0,1/6)$ such that $\eps_n \recht 0$, $\sum_n \eps_n^2 = +\infty$ and such that the function $F$ defined by \eqref{eq.def-F} has the property that the associated $1$-cocycle $c : G \recht \ell^2(G) : c_g(k) = F(k) - F(g^{-1}k)$ satisfies
$$\|c_{g_n}\|_2^2 \leq \frac{1}{16} \log(1+n)$$
for all $n \geq 0$.

Define the probability measures $\mu_k$ on $\{0,1\}$ given by $\mu_k(0) = F(k) + 1/2$ and note that $1/2 \leq \mu_k(0) \leq 2/3$ for all $k \in G$. Consider the associated Bernoulli action $G \actson (X,\mu)$, which is nonsingular because of \eqref{eq.easier-nonsingular}. Then,
$$\sum_{g \in \Lambda} \exp\bigl(-16 \|c_g\|_2^2\bigr) \geq \sum_{n=0}^\infty \exp\bigl(-\log(1+n)\bigr) = + \infty \; .$$
It follows from Proposition \ref{prop.conservative} that the action $\Lambda \actson (X,\mu)$ is conservative. By construction, for every $k \in G$, there is at most one $A_n$ that intersects $\Lambda k$. It then follows from Proposition \ref{prop.inessential-orbit} that $\Lambda k \subset G$ is $1/2$-inessential, for every $k \in G$. So by Proposition \ref{prop.inessential-type-III1}, the action $\Lambda \actson (X,\mu)$ is ergodic and of type III$_1$, and has a weakly mixing Maharam extension. A fortiori, the same holds for $G \actson (X,\mu)$.

{\bf Proof in case 2.} In case 2, $G$ also admits a copy of $\Z$ as a finite index \emph{normal} subgroup. Denote $\kappa = [G:\Z]$ and fix $g_1,\ldots,g_\kappa$ such that $G$ is the disjoint union of the $g_i \Z$. Define the function
$$F_0 : \Z \recht (0,1) : F_0(n) = \begin{cases} \frac{1}{2} &\;\;\text{if $n \leq 3$,}\\
\frac{1}{2} + \frac{1}{\sqrt{n \log(n)}} &\;\;\text{if $n \geq 4$,}\end{cases}$$
and then define the function $F : G \recht (0,1)$ given by $F(g_i n) = F_0(n)$ for all $i \in \{1,\ldots,\kappa\}$ and $n \in \Z$. For every $g \in G$, define the function $c_g : G \recht \R$ given by $c_g(h) = F(h) - F(g^{-1}h)$.

Since $\sum_{n=4}^k (n\log(n))^{-1}$ grows like $\log(\log(k))$, it follows from Lemma \ref{lem.translate-function} below that for every $k \in \Z$, the function $F_0 - k \cdot F_0$ belongs to $\ell^2(\Z)$ and that $\|F_0 - k \cdot F_0\|_2^2 / \log(|k|)$ tends to zero as $|k| \recht \infty$ in $\Z$. It then also follows that $c_g \in \ell^2(G)$ for every $g \in G$ and that $\|c_k\|_2^2 / \log(|k|)$ tends to zero when $k$ tends to infinity in $\Z$. Defining the probability measures $\mu_h$ on $\{0,1\}$ given by $\mu_h(0) = F(h)$, the associated Bernoulli action $G \actson (X,\mu)$ is nonsingular and essentially free.
%
%
%
Applying Proposition \ref{prop.Bernoulli-Z-generic} to the left action $\Z \actson G$, it follows that $\Z \actson (X,\mu)$ is ergodic and of type III$_1$, and has a weakly mixing Maharam extension. A fortiori, the same holds for $G \actson (X,\mu)$.

To conclude the proof of Theorem \ref{thm.amenable-III1}, let $G$ be an arbitrary amenable group. Applying the proof of case~1 to the amenable group $G \times \Z$ with the infinite subgroup $G \times \{0\}$ of infinite index, we find for every $(g,n) \in G \times \Z$ a probability measure $\mu_{(g,n)}$ on $\{0,1\}$ such that the Bernoulli action $G \actson (X,\mu) = \prod_{(g,n) \in G \times \Z} (\{0,1\},\mu_{(g,n)})$ is nonsingular and satisfies all the conclusions of the theorem. Defining $X_0 = \prod_{n \in \Z} \{0,1\}$ and $\mu_g = \prod_{n \in \Z} \mu_{(g,n)}$ for every $g \in G$, the Bernoulli action $G \actson (X,\mu)$ can be identified with the Bernoulli action $G \actson \prod_{g \in G} (X_0,\mu_g)$. This concludes the proof of Theorem \ref{thm.amenable-III1}.

\begin{lemma}\label{lem.translate-function}
Let $a_0 \geq a_1 \geq a_2 \geq \cdots$ be a decreasing sequence of strictly positive real numbers. Let $\lambda > 0$ and $n_0 \in \Z$. Define the function
$$F : \Z \recht (0,+\infty) : F(n) = \begin{cases} \lambda + a_{n-n_0} &\;\; \text{if $n \geq n_0$,}\\
\lambda &\;\;\text{if $n < n_0$.}\end{cases}$$
For every $k \in \Z$, define the function $c_k : \Z \recht \R : c_k(n) = F(n) - F(n-k)$. Then, $c_k \in \ell^2(\Z)$ and
$$\sum_{n=0}^{|k|-1} a_n^2 \; \leq \; \|c_k\|_2^2 \; \leq \; 2 \sum_{n=0}^{|k|-1} a_n^2 \quad\text{for every $k \in \Z$.}$$
\end{lemma}
\begin{proof}
Changing $\lambda$ or $n_0$ does not change the value of $\|c_k\|_2$, so that we may assume that $\lambda = 0$ and $n_0=0$. Fix $k \geq 1$. For every $n_1 \geq k$, we have
$$\sum_{n=-\infty}^{n_1} |c_k(n)|^2  = \sum_{n=0}^{k-1} a_n^2 + \sum_{n=k}^{n_1} (a_{n-k} - a_n)^2 \; ,$$
so that $\|c_k\|_2^2 \geq \sum_{n=0}^{k-1} a_n^2$ and
\begin{align*}
\sum_{n=-\infty}^{n_1} |c_k(n)|^2 & \leq \sum_{n=0}^{k-1} a_n^2 + \sum_{n=k}^{n_1} |a_{n-k}^2 - a_n^2|
= \sum_{n=0}^{k-1} a_n^2 + \sum_{n=k}^{n_1} (a_{n-k}^2 - a_n^2) \\
& = \sum_{n=0}^{k-1} a_n^2 + \sum_{n=0}^{k-1} a_n^2 - \sum_{n=n_1-k+1}^{n_1} a_n^2  \leq 2 \sum_{n=0}^{k-1} a_n^2 \; .
\end{align*}
Since this holds for all $n_1 \geq k$, we find that $c_k \in \ell^2(\Z)$ and
$$\|c_k\|_2^2 \leq 2 \sum_{n=0}^{k-1} a_n^2 \; .$$
Since $c_0 = 0$ and $\|c_{-k}\|_2 = \|c_k\|_2$, the lemma is proven.
\end{proof}

\subsection{Proof of Corollary \ref{cor.explicit-Z}}

It suffices to note that each of the transformations $T^{a_1} \times \cdots \times T^{a_k}$ can be viewed as a Bernoulli action associated with some free action $\Z \actson I$ having finitely many orbits. Since $\sum_{n=n_0}^k (n \log(n))^{-1}$ grows like $\log(\log(k))$, it follows from Lemma \ref{lem.translate-function} that the associated $1$-cocycle $c : \Z \recht \ell^2(I)$ satisfies $\lim_{|k| \recht \infty} \|c_k\|_2^2 / \log (|k|) = 0$. By Proposition \ref{prop.Bernoulli-Z-generic}, the transformation $T^{a_1} \times \cdots \times T^{a_k}$ is ergodic and of type III$_1$, and has a weakly mixing Maharam extension.

\subsection{Proof of Corollary \ref{cor.explicit-Z-power-dissipative}}\label{sec.proof-last-cor}

By Lemma \ref{lem.translate-function}, the associated $1$-cocycle $c : \Z \recht \ell^2(\Z)$ defined by \eqref{eq.one-cocycle-c} satisfies
$$\frac{1}{36} \sum_{n=1}^{|k|} \frac{1}{n} \; \leq \; \|c_k\|_2^2 \; \leq \; \frac{1}{18} \sum_{n=1}^{|k|} \frac{1}{n}$$
so that
$$\frac{1}{36} \log(1+|k|) \; \leq \; \|c_k\|_2^2 \; \leq \; \frac{1}{18} (1 + \log |k|)$$
whenever $|k|\geq 2$. It follows that
$$\sum_{k \in \Z} \exp(-16 \|c_k\|_2^2) \geq \sum_{k=2}^\infty \exp\bigl(-\frac{16}{18}(1+\log(k))\bigr) = \exp(-8/9) \sum_{k=2}^\infty \frac{1}{k^{8/9}} = +\infty \; .$$
Since $1/3 \leq \mu_k(0) \leq 2/3$, it follows from Proposition \ref{prop.Bernoulli-Z-generic} that
$T$ is ergodic, of type III$_1$, with weakly mixing Maharam extension.

Write $m=73$. The $m$-fold power of $T$ is a Bernoulli action associated with $\Z \actson I$, where $I$ is the disjoint union of $m$ copies of $\Z$. The associated $1$-cocycle $d : \Z \recht \ell^2(I)$ satisfies $\|d_k\|_2^2 = m \, \|c_k\|_2^2$ for every $k \in \Z$. Therefore,
\begin{align*}
\sum_{k \in \Z} \exp\bigl(-\frac{1}{2} \|d_k\|_2^2\bigr) &= 1 + 2 \sum_{k=1}^\infty \exp\bigl(-\frac{m}{2} \|c_k\|_2^2\bigr) \\
& \leq 1 + 2 \sum_{k=1}^\infty \exp\bigl(-\frac{m}{72} \log(1+k)\bigr) \\
& = 1 + 2 \sum_{k=2}^\infty \frac{1}{k^{m/72}} < +\infty \; .
\end{align*}
So by Proposition \ref{prop.conservative}, the $m$-fold power of $T$ is dissipative.

\section{Nonsingular Bernoulli actions of the free groups}\label{sec.bernoulli-free-group}

Concretizing the construction in the proof of Theorem \ref{thm.positive-betti-III1} in the special case of a free product group $G = \Lambda * \Z$, we obtain the following wide range of nonsingular Bernoulli actions. As we explain in Example \ref{ex.free-groups}, this provides nonsingular Bernoulli actions of type III$_\lambda$ for any $0 < \lambda < 1$ and this provides strongly ergodic nonsingular Bernoulli actions whose orbit equivalence relation can have any prescribed Connes invariant.

\begin{proposition}\label{prop.free-group-prescribed-invariants}
Let $G = \Lambda * \Z$ be any free product of an infinite group $\Lambda$ and the group of integers $\Z$. Define $W \subset G$ as the set of reduced words whose last letter is a strictly positive element of $\Z$. Let $\mu_0$ and $\mu_1$ be Borel probability measures on a standard Borel space $X_0$. Assume that $\mu_0 \sim \mu_1$ and that $\mu_0,\mu_1$ are not supported on a single atom.

The Bernoulli action $G \actson (X,\mu)$ with $(X,\mu) = \prod_{g \in G} (X_0,\mu_g)$ and
$$\mu_g = \begin{cases} \mu_1 &\;\;\text{if $g \in W$,}\\ \mu_0 &\;\;\text{if $g \not\in W$,}\end{cases}$$
is nonsingular, essentially free, ergodic and nonamenable in the sense of Zimmer.

Denote by $T = d\mu_1 / d\mu_0$ the Radon-Nikodym derivative. Define $\tau(T)$ as the weakest topology on $\R$ that makes the map
\begin{equation}\label{eq.map-pi}
\pi : \R \recht \cU(L^\infty(X_0,\mu_0)) : \pi(t) = (x \mapsto T(x)^{it})
\end{equation}
continuous, where $\cU(L^\infty(X_0,\mu_0))$ is equipped with the strong topology. We say that $T$ is almost periodic if there exists a countable subset $S \subset \R_*^+$ such that $T(x) \in S$ for a.e.\ $x \in X_0$. In that case, we denote by $\Sd(T)$ the subgroup of $\R_*^+$ generated by the smallest such $S \subset \R_*^+$.

\begin{enumlist}
\item The type of $G \actson (X,\mu)$ is determined as follows: the action is of type II$_1$ if and only if $T(x)=1$ for a.e.\ $x \in X_0$; the action is of type III$_\lambda$ with $0 < \lambda < 1$ if and only if the essential range of $T$ generates the subgroup $\lambda^\Z < \R_*^+$; and the action is of type III$_1$ if and only if the essential range of $T$ generates a dense subgroup of $\R_*^+$.

\item If $\Lambda$ is nonamenable, the action $G \actson (X,\mu)$ is strongly ergodic (in the sense of \cite{Sc79}). Then, the $\tau$-invariant of the orbit equivalence relation $\cR$ of $G \actson (X,\mu)$ (in the sense of \cite[Definition 2.6]{HMV17}) equals $\tau(T)$. In particular, $\cR$ is almost periodic (in the sense of \cite[Section 5]{HMV17}) if and only if $T$ is almost periodic and in that case, $\Sd(\cR) = \Sd(T)$.

\item If $\Lambda$ has infinite conjugacy classes and is non inner amenable, then the crossed product factor $M = L^\infty(X,\mu) \rtimes G$ is full and its $\tau$-invariant (in the sense of \cite{Co74}) equals $\tau(T)$. Also, $M$ is almost periodic (in the sense of \cite{Co74}) if and only if $T$ is almost periodic and in that case, $\Sd(M) = \Sd(T)$.
\end{enumlist}
\end{proposition}

For a Bernoulli action $G \actson (X,\mu)$ as in Proposition \ref{prop.free-group-prescribed-invariants}, the weak mixing of the Maharam extension and the stable type, i.e.\ the type of a diagonal action $G \actson (Y \times X,\eta \times \mu)$ given a pmp action $G \actson (Y,\eta)$, are discussed in Proposition \ref{prop.stable-type-concrete-examples} below.

Before proving Proposition \ref{prop.free-group-prescribed-invariants}, we provide the following concrete examples.

\begin{example}\label{ex.free-groups}
We use the same notations as in the formulation of Proposition \ref{prop.free-group-prescribed-invariants}.
\begin{enumlist}
\item Take $0 < \lambda < 1$ and put $X_0 = \{0,1\}$ with $\mu_0(0) = (1+\lambda)^{-1}$ and $\mu_1(0) = \lambda (1+\lambda)^{-1}$. It follows that $G \actson (X,\mu)$ is of type III$_\lambda$. So all free product groups $G = \Lambda * \Z$ with $\Lambda$ infinite admit nonsingular, essentially free, ergodic Bernoulli actions of type III$_\lambda$. Note that by \cite[Corollary 3.3]{DL16}, the group $\Z$ does not admit nonsingular Bernoulli actions of type III$_\lambda$, at least under the assumption that all $\mu_n$, $n < 0$, are identical.

\item Using the construction of \cite[Section 5]{Co74}, we obtain the following examples of strongly ergodic, nonsingular Bernoulli actions whose orbit equivalence relation has an arbitrary countable dense subgroup of $\R_*^+$ as $\Sd$-invariant or has any topology coming from a unitary representation of $\R$ as $\tau$-invariant. This holds for any free product group $G = \Lambda * \Z$ with $\Lambda$ nonamenable, and in particular for any free group $\F_n$ with $3 \leq n \leq +\infty$. So this provides an answer to \cite[Problem 3]{HMV17}.

    Let $\eta$ be any nonzero finite Borel measure on $\R_*^+$ with $\int_{\R_*^+} x \, d\eta(x) < \infty$. Define $X_0 = \R_*^+ \times \{0,1\}$ and define the probability measures $\mu_0$ and $\mu_1$ on $X_0$ determined by
    \begin{align*}
    \kappa &= \int_{\R_*^+} (1+x) \, d\eta(x) \; ,\\
    \int_{X_0} F \, d\mu_0 &= \kappa^{-1} \int_{\R_*^+} \bigl( F(x,0)+ xF(x,1) \bigr) \, d\eta(x) \; , \\
    \int_{X_0} F \, d\mu_1 &= \kappa^{-1} \int_{\R_*^+} \bigl( x F(x,0)+ F(x,1) \bigr) \, d\eta(x) \; ,
    \end{align*}
    for all positive Borel functions $F$ on $X_0$. Then, $\mu_0 \sim \mu_1$ and the Radon-Nikodym derivative $T = d\mu_1 / d\mu_0$ is given by $T(x,0) = x$ and $T(x,1) = 1/x$ for all $x \in \R_*^+$.

    So when $\Lambda$ is nonamenable, the nonsingular Bernoulli action associated with $\mu_0,\mu_1$ in Proposition \ref{prop.free-group-prescribed-invariants} is strongly ergodic and the $\tau$-invariant of the orbit equivalence relation is the weakest topology on $\R$ that makes the map
    $$\R \recht \cU(L^\infty(\R_*^+,\eta)) : t \mapsto (x \mapsto x^{it})$$
    continuous. Varying $\eta$, it follows that any topology on $\R$ induced by a unitary representation of $\R$ arises as the $\tau$-invariant of the orbit equivalence relation of a strongly ergodic, nonsingular Bernoulli actions of a free product $G = \Lambda * \Z$ with $\Lambda$ nonamenable.

    In particular, taking an atomic measure $\eta$, we obtain strongly ergodic, nonsingular Bernoulli actions of $G = \Lambda * \Z$ with any prescribed $\Sd$-invariant. More concretely, when $S < \R_*^+$ is a given countable dense subgroup, we enumerate $S \cap (0,1) = \{t_n \mid n \geq 1\}$ and define the finite atomic measure $\eta$ on $\R_*^+$ given by
    $$\eta = \sum_{n=1}^\infty \frac{1}{2^n (1 + t_n)} \, \delta_{t_n} \; .$$
    The orbit equivalence relation of $G \actson (X,\mu)$ is then almost periodic with $\Sd$-invariant equal to $S$.
\end{enumlist}
\end{example}

\begin{proof}[{Proof of Proposition \ref{prop.free-group-prescribed-invariants}}]
Since $\Lambda W = W$, the action $\Lambda \actson (X,\mu)$ is a probability measure preserving Bernoulli action. Denote by $a \in \Z$ the generator $a=1$. The measure $a^{-1} \cdot \mu$ given by $(a^{-1} \cdot \mu)(\cU) = \mu(a \cdot \cU)$ equals the product measure
$$a^{-1} \cdot \mu = \prod_{g \in G} \mu_{ag} \; .$$
Since $a^{-1} W \triangle W = \{e\}$, we get that $a^{-1} \cdot \mu \sim \mu$ and that
$$\frac{d(a^{-1} \cdot \mu)}{d\mu}(x) = \frac{d\mu_1}{d\mu_0}(x_e) \; .$$
So, $a$ acts nonsingularly on $(X,\mu)$ and the Radon-Nikodym cocycle is given by $\om(a,x) = T(x_e)$. It follows that $G \actson (X,\mu)$ is nonsingular and essentially free.

To prove the ergodicity and to determine the type of $G \actson (X,\mu)$, consider the Maharam extension $G \actson (X \times \R, \mu \times \nu)$ given by \eqref{eq.maharam}. Let $Q \in L^\infty(X \times \R)$ be a $G$-invariant function. Since $\Lambda \actson (X,\mu)$ is measure preserving and ergodic, it follows that $Q(x,s) = P(s)$, where $P \in L^\infty(\R)$ is invariant under translation by $t$ for every $t$ in the essential range of one of the maps $x \mapsto \log(\om(g,x))$, $g \in G$. The union of these essential ranges equals the subgroup of $\R_*^+$ generated by the essential range of $T$. So our statements about the ergodicity and the type of $G \actson (X,\mu)$ follow.

To prove that $G \actson (X,\mu)$ is nonamenable in the sense of Zimmer, denote by $\Lambda_1 < G$ the subgroup generated by $\Lambda$ and $a \Lambda a^{-1}$. Note that $\Lambda_1$ is the free product of these two subgroups. Both $\Lambda$ and $a \Lambda a^{-1}$ act on $(X,\mu)$ as a probability measure preserving Bernoulli action, although they do not preserve the same probability measure. In particular, the actions of $\Lambda$ and $a \Lambda a^{-1}$ on $(X,\mu)$ are conservative. Since the action of their free product $\Lambda_1$ is essentially free, it follows from \cite[Corollary F]{HV12} that $\Lambda_1 \actson (X,\mu)$ is nonamenable in the sense of Zimmer. A fortiori, $G \actson (X,\mu)$ is nonamenable.

Now assume that $\Lambda$ is nonamenable. Since $\Lambda \actson (X,\mu)$ is a probability measure preserving Bernoulli action, the action $\Lambda \actson (X,\mu)$ is strongly ergodic. A fortiori, $G \actson (X,\mu)$ is strongly ergodic. The same argument as in \cite[Theorem 6.4]{HMV17} gives us that the $\tau$-invariant of the orbit equivalence relation $\cR(G \actson (X,\mu))$ is the weakest topology on $\R$ that makes the map in \eqref{eq.map-pi} continuous.

Finally assume that $\Lambda$ has infinite conjugacy classes and that $\Lambda$ is non inner amenable. Denote by $(u_g)_{g \in G}$ the canonical unitary operators in $M = L^\infty(X) \rtimes G$ and denote by $\vphi$ the canonical faithful normal state on $M$ given by $\vphi(F) = \int_X F(x) d\mu(x)$ and $\vphi(F u_g) = 0$ for all $F \in L^\infty(X)$ and $g \in G \setminus \{e\}$. Denote by $\cH$ the Hilbert space completion of $M$ w.r.t.\ the scalar product given by $\langle c,d \rangle = \vphi(d^* c)$ for all $c,d \in M$. View $M \subset \cH$. Since the action $\Lambda \actson (X,\mu)$ is measure preserving, both left and right multiplication by $u_g$, $g \in \Lambda$, defines a unitary operator on $\cH$. To prove that the factor $M$ is full and that the same topology as above is the $\tau$-invariant of $M$, it suffices to prove that the unitary representation
$$\theta : \Lambda \recht \cU(\cH \ominus \C 1) : (\theta(g))(d) = u_g d u_g^*$$
does not weakly contain the trivial representation of $\Lambda$.

But $\theta$ is the direct sum of the subrepresentations $\theta_i$ on $\cH_i$ where $\cH_1$ is the closed linear span of $\{u_g F \mid g \in G, \int_X F \, d\mu = 0\}$, where $\cH_2$ is the closed linear span of $\{ u_g \mid g \in G \setminus \Lambda\}$, and where $\cH_3$ is the closed linear span of $\{ u_g \mid g \in \Lambda \setminus \{e\}\}$. Because $\Lambda \actson (X,\mu)$ is a probability measure preserving Bernoulli action, the representation $\theta_1$ is a multiple of the regular representation of $\Lambda$. Since $G$ is the free product of $\Lambda$ and $\Z$, also $\theta_2$ is a multiple of the regular representation of $\Lambda$. Since $\Lambda$ is nonamenable, $\theta_1$ and $\theta_2$ do not weakly contain the trivial representation of $\Lambda$. Finally, $\theta_3$ does not weakly contain the trivial representation of $\Lambda$ because $\Lambda$ has infinite conjugacy classes and $\Lambda$ is not inner amenable.
\end{proof}

\begin{proposition}\label{prop.stable-type-concrete-examples}
Let $G = \Lambda * \Z$ be any free product of an infinite group $\Lambda$ and the group of integers $\Z$. Let $G \actson (X,\mu)$ be a Bernoulli action as in Proposition \ref{prop.free-group-prescribed-invariants}. Choose an ergodic pmp action $G \actson (Y,\eta)$. Then, the diagonal action $G \actson Y \times X$ is ergodic and its type is determined as follows.

Using the same notations as in Proposition \ref{prop.free-group-prescribed-invariants}, denote $T = d\mu_1 / d\mu_0$. Denote by $L < \R$ the subgroup generated by the essential range of the map $X_0 \times X_0 \recht \R : (x,x') \mapsto \log(T(x)) - \log(T(x'))$.

\begin{enumlist}
\item\label{point.one} If $L = \{0\}$, then $\mu$ is $G$-invariant and the actions $G \actson X$ and $G \actson Y \times X$ are of type II$_1$.

\item\label{point.two} If $L < \R$ is dense, then the Maharam extension of $G \actson (X,\mu)$ is weakly mixing and the diagonal action $G \actson Y \times X$ is of type III$_1$.

\item\label{point.three} If $L = a \Z$, take the unique $b \in [0,a)$ such that $\log(T(x)) \in b + a \Z$ for a.e.\ $x \in X_0$. Denote by $\pi : G \recht \Z$ the unique homomorphism given by $\pi(g) = 0$ if $g \in \Lambda$ and $\pi(n) = n$ if $n \in \Z$. The set
    \begin{equation}\label{eq.my-subgroup-with-k0}
    \begin{split}
    H = \{k \in \Z \mid \; &\text{there exists a Borel map $V : Y \recht \R/a\Z$ s.t.}\\ & \text{$V(g \cdot y) = V(y) + k \pi(g) b$ for all $g \in G$ and a.e.\ $y \in Y$}\}
    \end{split}
    \end{equation}
    is a subgroup of $\Z$. Write $H = k_0 \Z$ with $k_0 \geq 0$.
    If $k_0 = 0$, the action $G \actson Y \times X$ is of type III$_1$. If $k_0 \geq 1$, the action $G \actson Y \times X$ is of type III$_\lambda$ with $\lambda = \exp(-a/k_0)$.

\item\label{point.four} If $L = a \Z$ and $b \in [0,a)$ is defined as in \ref{point.three}, then the following holds.
\begin{itemlist}
\item If $b$ is of finite order $k_1 \geq 1$ in $\R/a\Z$ (with the convention that $k_1=1$ if $b=0$), varying the action $G \actson (Y,\eta)$, the possible types of $G \actson Y \times X$ are III$_\lambda$ with $\lambda = \exp(-a/k_0)$ where $k_0 \geq 1$ is an integer dividing $k_1$. Given such a $k_0$, this type is realized by taking the transitive action of $G$ on $Y = \Z / (k_1/k_0)\Z$ given by $g \cdot y = y + \pi(g)$, or any other pmp action $G \actson Y$ that is induced from a weakly mixing pmp action of the finite index normal subgroup $\pi^{-1}((k_1/k_0)\Z) < G$.

\item If $b$ is of infinite order in $\R/a\Z$, varying the action $G \actson (Y,\eta)$, the possible types of $G \actson Y \times X$ are III$_1$ and III$_\lambda$ with $\lambda = \exp(-a/k_0)$ where $k_0 \geq 1$ is any integer. Given $k_0$, the latter is realized by taking $Y = \R/(a/k_0)\Z$ and $g \cdot y = y + \pi(g)b$, while the former is realized by taking $G \actson (Y,\eta)$ to be the trivial action, or any other weakly mixing action.
\end{itemlist}
\end{enumlist}
\end{proposition}

%
%
%
%
%

By varying the probability measures $\mu_0$ and $\mu_1$ in the construction of Proposition \ref{prop.free-group-prescribed-invariants}, all values of $0 \leq b < a$ in Proposition \ref{prop.stable-type-concrete-examples} occur; see Example \ref{ex.all-values-a-b}.

\begin{proof}
Fix $G \actson (X,\mu)$ as in Proposition \ref{prop.free-group-prescribed-invariants} and fix an arbitrary ergodic pmp action $G \actson (Y,\eta)$. Since $\Lambda \actson (X,\mu)$ is a pmp Bernoulli action, a $\Lambda$-invariant element of $L^\infty(Y \times X)$ belongs to $L^\infty(Y) \ot 1$. It follows that $G \actson Y \times X$ is ergodic.

Define $L < \R$ as in the formulation of the proposition. If $L = \{0\}$, we have that $T$ is constant a.e. Since $\int_{X_0} T(x) d\mu_0(x) = 1$, this constant must be $1$. So, $T(x) = 1$ for a.e.\ $x \in X_0$. This means that $\mu_0 = \mu_1$, so that $G \actson (X,\mu)$ is a pmp Bernoulli action. This proves point \ref{point.one}.

To prove the remaining points of the proposition, let $Q \in L^\infty(Y \times X \times \R)$ be a $G$-invariant element for the diagonal action of $G \actson Y$ and the Maharam extension $G \actson X \times \R$ of $G \actson X$. A fortiori, $Q$ is $\Lambda$-invariant. Since $\Lambda \actson (X,\mu)$ is a pmp Bernoulli action, it follows that $Q \in L^\infty(Y) \ovt 1 \ovt L^\infty(\R)$.

As in the proof of Theorem \ref{thm.positive-betti-III1}, it follows that $Q(y,x,s) = P(y,s)$, where $P \in L^\infty(Y \times \R)$ satisfies
\begin{equation}\label{eq.my-inv-Q}
P(g \cdot y, s+\log(\om(g,x))) = P(y,s) \quad\text{for all $g \in G$ and a.e.\ $(y,x,s) \in Y \times X \times \R$.}
\end{equation}
Note that $L$ equals the subgroup of $\R$ generated by the essential ranges of the maps
$$X \times X \recht \R : (x,x') \mapsto \log(\om(g,x)) - \log(\om(g,x')) \; , \; g \in G \; .$$
It then follows from \eqref{eq.my-inv-Q} that $P(y,s+t) = P(y,s)$ for all $t \in L$ and a.e.\ $(y,s) \in Y \times \R$.

If $L < \R$ is dense, we conclude that $Q \in L^\infty(Y) \ot 1 \ot 1$ and thus, by ergodicity of $G \actson Y$, that $Q$ is constant a.e., so that $G \actson Y \times X \times \R$ is ergodic. This means that $G \actson Y \times X$ is of type III$_1$. Since $G \actson (Y,\eta)$ was an arbitrary ergodic pmp action, it follows that the Maharam extension $G \actson X \times \R$ is weakly mixing. This proves point \ref{point.two}.

Next assume that $L = a \Z$ with $a > 0$ and take the unique $0 \leq b < a$ such that $\log(T(x)) \in b + a \Z$ for a.e.\ $x \in X_0$. Denote by $\pi : G \recht \Z$ the unique homomorphism given by $\pi(g) = 0$ if $g \in \Lambda$ and $\pi(n) = n$ if $n \in \Z$. Since $\om(g,x) = 1$ for all $g \in \Lambda$ and $\om(1,x) = T(x_e)$, it follows that $\log(\om(g,x)) \in \pi(g) b + a\Z$ for all $g \in G$ and a.e.\ $x \in X$. We conclude that an element $Q \in L^\infty(Y \times X \times \R)$ is $G$-invariant if and only if $G(y,x,s) = P(y,s)$ where $P \in L^\infty(Y \times \R/a\Z)$ is invariant under the action $G \actson Y \times \R/a\Z$ given by $g \cdot (y,s) = (g \cdot y, \pi(g)b + s)$.

If $k \in \Z$ and $V : Y \recht \R / a \Z$ is a Borel map satisfying $V(g \cdot y) = V(y) + k \pi(g) b$ for all $g \in G$ and a.e.\ $y \in Y$, the map $P(y,s) = \exp(2\pi i (V(y) - ks)/a)$ is $G$-invariant. Using a Fourier decomposition for $\R / a \Z \cong \widehat{\Z}$, it follows that these functions $P$ densely span the space of all $G$-invariant functions in $L^2(Y \times \R/a \Z)$. Define $H < \Z$ as in \eqref{eq.my-subgroup-with-k0}. If $H = \{0\}$, it follows that $L^\infty(Y \times X \times \R)^G = \C 1$ and that $G \actson Y \times X$ is of type III$_1$. When $H = k_0 \Z$ with $k_0 \geq 1$, we identified $L^\infty(Y \times X \times \R)^G$ with $L^\infty(\R/(a/k_0)\Z)$ and it follows that $G \actson Y \times X$ is of type III$_\lambda$ with $\lambda = \exp(-a/k_0)$. This concludes the proof of point \ref{point.three}.

To prove point \ref{point.four}, first assume that $b$ is of finite order $k_1$ in $\R / a \Z$. Using the map $V(y) = 0$ for all $y \in Y$, it follows that $k_1$ belongs to the subgroup $H < \Z$ defined in \eqref{eq.my-subgroup-with-k0}. Therefore, $k_0$ must divide $k_1$. Conversely, assume that $k_0 \geq 1$ divides $k_1$ and that $G \actson Y$ is induced from a weakly mixing pmp action of $G_0 := \pi^{-1}((k_1/k_0)\Z)$ on $Y_0$. Denote by $H < \Z$ the subgroup defined in \eqref{eq.my-subgroup-with-k0}. We have to prove that $H = k_0 \Z$. If $k \in \Z$ and $V : Y \recht \R / a \Z$ is a Borel function satisfying $V(g \cdot y) = V(y) + k \pi(g) b$, it follows that $V$ is invariant under $\pi^{-1}(k_1 \Z)$. Since $G_0 \actson Y_0$ is weakly mixing, $G_0$ is normal in $G$ and $\pi^{-1}(k_1 \Z) < G_0$ has finite index, it follows that $V$ is $G_0$-invariant. This forces $k$ to be a multiple of $k_0$. So, $H \subset k_0 \Z$. By construction of the induced action, there is a Borel map $W : Y \recht G/G_0$ satisfying $W(g \cdot y) = g W(y)$. Identifying $G/G_0$ with $\Z / ((k_1/k_0)\Z)$ through $\pi$ and composing $W$ with the map
$$\Z / ((k_1/k_0)\Z) \recht \R / a \Z : n \mapsto k_0 n b \; ,$$
we have found a Borel map $V : Y \recht \R / a \Z$ satisfying $V(g \cdot y) = V(y) + k_0 \pi(g) b$. So, $k_0 \in H$ and the equality $H = k_0 \Z$ follows. By point \ref{point.three}, the action $G \actson Y \times X$ is of type III$_\lambda$ with $\lambda = \exp(-a / k_0)$.

Finally assume that $b$ is of infinite order in $\R/ a\Z$. When $G \actson (Y,\eta)$ is weakly mixing, the subgroup of $H < \Z$ defined in \eqref{eq.my-subgroup-with-k0} is trivial, so that $G \actson Y \times X$ is of type III$_1$. When $Y = \R /((a/k_0)\Z$ with $g \cdot y = y + \pi(g)b$, one checks that $H = k_0 \Z$, so that $G \actson Y \times X$ is of type III$_\lambda$ with $\lambda = \exp(-a/k_0)$.
\end{proof}

\begin{remark}\label{ex.all-values-a-b}
Given $0 < b < a$, define the probability measures $\mu_0$ and $\mu_1$ on $\{0,1\}$ given by
$$\mu_0(0) = \frac{1-\exp(-b)}{1-\exp(-a)} \quad\text{and}\quad \mu_1(0) = \frac{1-\exp(b-a)}{1-\exp(-a)} \; .$$
Denote $T = d\mu_1 / d\mu_0$. We get that $T(0) = \exp(b)$ and $T(1) = \exp(b-a)$. So, the map $(x,x') \mapsto \log(T(x)) - \log(T(x'))$ generates the subgroup $a \Z < \R$ and $\log(T(x)) \in b + a \Z$ for all $x \in \{0,1\}$.

Given $a > 0$ and $b=0$, define the probability measures $\mu_0$ and $\mu_1$ on $\{0,1,2\}$ given by
\begin{align*}
\mu_0(0) &= \frac{1}{2} \quad , &\;\; \mu_0(1) &= \frac{1}{2(1 + \exp(a))} \quad , &\;\; \mu_0(2) &= \frac{\exp(a)}{2(1 + \exp(a))} \quad ,\\
\mu_1(0) &= \frac{1}{2} \quad , &\;\; \mu_1(1) &= \frac{\exp(a)}{2(1 + \exp(a))} \quad , &\;\; \mu_1(2) &= \frac{1}{2(1 + \exp(a))} \quad .
\end{align*}
The range of $T = d\mu_1 / d\mu_0$ equals $\{1,\exp(a),\exp(-a)\}$. Therefore, the range of the map $(x,x') \mapsto \log(T(x)) - \log(T(x'))$ generates the subgroup $a \Z < \R$ and $\log(T(x)) \in a \Z$ for all $x \in \{0,1, 2\}$.

So all values $0 \leq b < a$ really occur in Proposition \ref{prop.stable-type-concrete-examples}.

This means that given any $0 < \lambda < 1$, Proposition \ref{prop.stable-type-concrete-examples} provides concrete examples of nonsingular, weakly mixing Bernoulli actions $G \actson (X,\mu)$ of a free product group $G = \Lambda * \Z$ such that the type of $G \actson (Y \times X,\eta \times \mu)$ ranges over III$_\mu$ with $\mu \in \{1\} \cup \{\lambda^{1/k} \mid k \geq 1\}$.

Given any $0 < \lambda < 1$ and an integer $k_1 \geq 1$, Proposition \ref{prop.stable-type-concrete-examples} also provides concrete examples of nonsingular, weakly mixing Bernoulli actions $G \actson (X,\mu)$ such that the type of a diagonal action $G \actson (Y \times X,\eta \times \mu)$ ranges over III$_\mu$ with $\mu \in \{\lambda^{1/k} \mid k \geq 1 \; , \; k|k_1\}$. In particular, we find nonsingular Bernoulli actions of stable type III$_\lambda$.
\end{remark}

In Corollary \ref{cor.explicit-Z-power-dissipative}, we constructed explicit nonsingular Bernoulli actions $\Z \actson (X,\mu)$ of type III$_1$ such that the $m$-th power diagonal action $\Z \actson (X^m,\mu^m)$ is dissipative. However, as we explain now, this phenomenon does not always occur for nonamenable groups.

Let $G$ be a nonamenable group, $G \actson I$ a free action and $F : I \recht (0,1)$ a function satisfying \eqref{eq.conds-F}. Consider the associated nonsingular Bernoulli action $G \actson (X,\mu)$ and the $1$-cocycle $c : G \recht \ell^2(I)$ given by \eqref{eq.one-cocycle-c}. If the $1$-cocycle is not proper, meaning that there exists a $\kappa > 0$ such that $\|c_g\|_2 \leq \kappa$ for infinitely many $g \in G$, it follows from Proposition \ref{prop.conservative} that $G \actson (X,\mu)$ and all its diagonal actions $G \actson (X^m,\mu^m)$ are conservative.

So, if the group $G$ has no proper $1$-cocycles into $\ell^2(G)$, e.g.\ because $G$ does not have the Haagerup property, then all its nonsingular Bernoulli actions are conservative.

On the other hand, the free group $\F_2$ admits proper $1$-cocycles into $\ell^2(\F_2)$. We use this to construct the following peculiar example of a nonsingular Bernoulli action of $\F_2$. In Proposition \ref{prop.free-group-dissipative}, we use a $1$-cocycle with faster growth to give an example of a dissipative Bernoulli action of $\F_2$.

\begin{proposition}\label{prop.concrete-free-group-1}
Let $G = \F_2$ be freely generated by the elements $a$ and $b$. Define the subset $W_a \subset \F_2$ consisting of all reduced words in $a,b$ that end with a strictly positive power of $a$. Similarly define $W_b \subset \F_2$ and put $W = \F_2 \setminus (W_a \cup W_b)$. The Bernoulli action $G \actson (X,\mu)$ with $(X,\mu) = \prod_{g \in G} (\{0,1\},\mu_g)$ and
$$\mu_g(0) = \begin{cases} 3/5 &\;\;\text{if $g \in W_a$,}\\ 2/5 &\;\;\text{if $g \in W_b$,} \\ 1/2 &\;\;\text{if $g \in W$,}\end{cases}$$
is nonsingular, essentially free, ergodic, nonamenable in the sense of Zimmer and of type III$_1$.

For every $g \in G \setminus \{e\}$, the transformation $x \mapsto g \cdot x$ is dissipative. For $m \geq 220$, the $m$-th power diagonal action $G \actson (X^m,\mu^m)$ is dissipative.
\end{proposition}

The stable type of the Bernoulli actions $\F_2 \actson (X,\mu)$ in Proposition \ref{prop.concrete-free-group-1} is discussed in Remark \ref{rem.stable-examples-weird}.

\begin{proof}
Denote $F : G \recht (0,1) : F(g) = \mu_g(0)$ and define $c_g(h) = F(h) - F(g^{-1} h)$. We find that
$$c_a = \frac{1}{10} \delta_a \quad\text{and}\quad c_b = - \frac{1}{10} \delta_b \; .$$
Since $c$ is a $1$-cocycle, it follows that $c_g \in \ell^2(G)$ for all $g \in G$. So, the action $G \actson (X,\mu)$ is nonsingular. Using the $1$-cocycle relation, we find that
$$c_{a^n} = \begin{cases} \dis \frac{1}{10} \sum_{k=1}^n \delta_{a^k} &\;\;\text{if $n \geq 1$,}\\ \dis  -\frac{1}{10} \sum_{k=n+1}^0 \delta_{a^k} &\;\;\text{if $n \leq -1$,} \\ \dis 0 &\;\;\text{if $n=0$,}\end{cases}
\qquad\text{and}\qquad
c_{b^n} = \begin{cases} \dis -\frac{1}{10} \sum_{k=1}^n \delta_{b^k} &\;\;\text{if $n \geq 1$,}\\ \dis \frac{1}{10} \sum_{k=n+1}^0 \delta_{b^k} &\;\;\text{if $n \leq -1$,} \\ \dis 0 &\;\;\text{if $n=0$.}\end{cases}
$$

When $g = a^{n_0} b^{m_1} a^{n_1} \cdots a^{n_{k-1}} b^{m_k} a^{n_k}$ is a reduced word, with $k \geq 0$, $n_0,n_k \in \Z$ and $n_i,m_j \in \Z \setminus \{0\}$, the $1$-cocycle relation implies that
\begin{equation}\label{eq.sum-cg}
c_g = c_{a^{n_0}} + a^{n_0} \cdot c_{b^{m_1}} + a^{n_0} b^{m_1} \cdot c_{a^{n_1}} + \cdots + a^{n_0} b^{m_1} a^{n_1} \cdots a^{n_{k-1}} b^{m_k} \cdot c_{a^{n_k}} \; .
\end{equation}
All the terms at the right hand side of \eqref{eq.sum-cg} are orthogonal, except two consecutive terms whose scalar product equals $1/100$ when $n_i \geq 1$ and $m_{i+1} \leq -1$, and also when $m_i \geq 1$ and $n_i \leq -1$. Denote by $|g|$ the word length of $g \in \F_2$. We conclude that
\begin{equation}\label{eq.norm-cg}
\|c_g\|_2^2 = \frac{1}{100} |g| + \frac{1}{50} \text{number of sign changes in the sequence $n_0,m_1,n_1,\ldots,m_k,n_k$} \; .
\end{equation}

Denote by $\om : G \times X \recht (0,\infty)$ the Radon-Nikodym cocycle. Define $\cF = \{a,a^{-1},b,b^{-1}\}$. Denote by $\lambda : G \recht \cU(\ell^2(G))$ the left regular representation.

Combining Lemma \ref{lem.lower-est-cocycle} with \eqref{eq.norm-cg} and Kesten's \cite{Ke58}, we find that
$$\sum_{g \in \cF} \int_X \sqrt{\om(g,x)} \, d\mu(x) \geq 4 \exp\bigl( - \frac{3}{500} \bigr) > 2 \sqrt{3} = \Bigl\| \sum_{g \in \cF} \lambda_g \Bigr\| \; .$$
So by Proposition \ref{prop.criterion-nonamenable}, the action $G \actson (X,\mu)$ is nonamenable in the sense of Zimmer.

When $g_0 \in G \setminus \{e\}$, there exist integers $\al,\be$ with $\al \geq 1$ and $\be \geq 0$ such that $|g_0^n| = \al |n| + \be$ for all $n \in \Z \setminus \{0\}$. It then follows from \eqref{eq.norm-cg} that
$$\sum_{n \in \Z} \exp\bigl(-\frac{1}{2} \, \|c_{g_0^n}\|_2^2 \bigr) \leq \sum_{n \in \Z} \exp\bigl(-\frac{1}{200} \, |g_0^n| \bigr)
\leq 1 + 2 \sum_{n=1}^\infty \exp\bigl(-\frac{\al}{200} n\bigr) < +\infty \; .$$
So by Proposition \ref{prop.conservative}, the transformation $x \mapsto g_0 \cdot x$ is dissipative.

Let $m \geq 220$. The $m$-th power diagonal action $G \actson (X^m,\mu^m)$ is a Bernoulli action whose corresponding $1$-cocycle $(c_{m,g})_{g \in G}$ satisfies $\|c_{m,g}\|_2^2 = m \, \|c_g\|_2^2$.
Define $B_n = \{g \in G \mid |g|=n\}$. For every $n \geq 1$, we have $|B_n| = 4 \cdot 3^{n-1}$. Therefore, using \eqref{eq.norm-cg}, we get that
$$\sum_{g \in G} \exp\bigl(- \frac{1}{2} \, \|c_{m,g}\|_2^2 \bigr) \leq \sum_{g \in G} \exp\bigl(-\frac{m}{200} \, |g|\bigr) =
1 + \sum_{n=1}^\infty \exp\bigl(- \frac{m}{200} \, n\bigr) \cdot 4 \cdot 3^{n-1} < +\infty$$
because $m > 200 \cdot \log 3$. It follows from Proposition \ref{prop.conservative} that the $m$-th power diagonal action $G \actson (X^m,\mu^m)$ is dissipative.

It remains to prove that $G \actson (X,\mu)$ is ergodic and of type III$_1$. Denote by $G \actson (X \times \R,\mu \times \nu)$ the Maharam extension given by \eqref{eq.maharam}. Let $Q \in L^\infty(X \times \R)$ be a $G$-invariant function. The main point is to prove that $Q \in 1 \ot L^\infty(\R)$.

Denote by $S_a \subset G$ the set of reduced words that start with a strictly positive power of $a$. Similarly define $S_{a^{-1}}$, $S_b$ and $S_{b^{-1}}$. Note that
$$\F_2 = \{e\} \sqcup S_a \sqcup S_{a^{-1}} \sqcup S_b \sqcup S_{b^{-1}} \; .$$
Whenever $U \subset G$, we denote $(X_U,\mu_U) = \prod_{g \in U} (\{0,1\},\mu_g)$ and we identify $(X,\mu) = (X_U \times X_{U^c} , \mu_U \times \mu_{U^c})$. Define $\Lambda = \langle b, a^{-1} b a \rangle$ and note that $\Lambda$ is freely generated by $b$ and $a^{-1} b a$. The concatenation $w v$ of a reduced word $w \in \Lambda$ and a reduced word $v \in S_a$ remains reduced. In particular, for all $w \in \Lambda$ and $v \in S_a$, the last letter of $w v$ equals the last letter of $v$. Therefore, the restriction of $F$ to $U := \Lambda S_a$ is $\Lambda$-invariant. It follows that $\Lambda \actson (X_U,\mu_U)$ is a probability measure preserving Bernoulli action.

We claim that the action $\Lambda \actson (X,\mu)$ is conservative. Whenever $k \geq 1$ and $n_i,m_j \geq 1$, the element
$$g = (a^{-1} b a)^{n_1} \, b^{m_1} \cdots (a^{-1} b a)^{n_k} \, b^{m_k} = a^{-1}  b^{n_1} a \; b^{m_1} \; a^{-1} b^{n_2}  a \; b^{m_2} \cdots a^{-1} b^{n_k} a \; b^{m_k}$$
belongs to $\Lambda$ and by \eqref{eq.norm-cg}, we have
$$\|c_g\|_2^2 = \frac{1}{100} \Bigl(2k + \sum_{i=1}^k (n_i + m_i) \Bigr) + \frac{2k-1}{50} < \frac{1}{100} \sum_{i=1}^k (n_i + m_i) + \frac{3k}{50} \; .$$
It follows that
\begin{align*}
\sum_{g \in \Lambda} \exp\bigr(-16 \|c_g\|_2^2\bigr) &\geq \sum_{k=1}^\infty \;\; \sum_{n_1,\ldots,n_k,m_1,\ldots,m_k = 1}^\infty \;\; \exp\bigl(-\frac{24}{25} k\bigr) \; \prod_{i=1}^k \exp\bigl(-\frac{4}{25}(n_i+m_i)\bigr) \\
&= \sum_{k=1}^\infty \exp\bigl(-\frac{24}{25} k\bigr) \; \left( \sum_{n=1}^\infty \exp\bigl(-\frac{4}{25} n\bigr) \right)^{2k} \\
&= \sum_{k=1}^\infty \left( \frac{\exp\bigl(-\frac{32}{25}\bigr)}{\bigl(1-\exp\bigl(-\frac{4}{25}\bigr)\bigr)^2}\right)^k \\
&=+\infty \; .
\end{align*}
From Proposition \ref{prop.conservative}, the claim that $\Lambda \actson (X,\mu)$ is conservative follows.

Since $\Lambda \actson (X_{U^c},\mu_{U^c})$ is a factor action of $\Lambda \actson (X,\mu)$, it is also conservative, as well as its Maharam extension $\Lambda \actson (X_{U^c} \times \R , \mu_{U^c} \times \nu)$. Since the action $\Lambda \actson (X_U,\mu_U)$ preserves the probability measure $\mu_U$, we can view $\Lambda \actson (X \times \R, \mu \times \nu)$ as the diagonal product of the mixing, probability measure preserving $\Lambda \actson (X_U,\mu_U)$ and the conservative $\Lambda \actson (X_{U^c} \times \R , \mu_{U^c} \times \nu)$. By \cite[Theorem 2.3]{SW81}, it follows that $Q \in L^\infty(X_{U^c} \times \R)$. In particular, $Q \in L^\infty(X_{S_a^c} \times \R)$.

We make the same reasoning for $S_{a^{-1}}$ and the group $\langle b, a b a^{-1} \rangle$, for $S_b$ and the group $\langle a , b^{-1} a b \rangle$ and for $S_{b^{-1}}$ and the group $\langle a, b a b^{-1} \rangle$. Since $S_a \cup S_{a^{-1}} \cup S_b \cup S_{b^{-1}} = G \setminus \{e\}$, it follows that $Q \in L^\infty(X_{\{e\}} \times \R)$.

We finally use the group $\Lambda = \langle a b a^{-1}, a^2 b a^{-2} \rangle$. We have $\Lambda \subset W$, so that $\Lambda \actson (X_\Lambda,\mu_\Lambda)$ is a probability measure preserving Bernoulli action. For all $k \geq 1$ and $n_i,m_j \geq 1$, we have that
\begin{align*}
g &= (a b a^{-1})^{n_1} \; (a^2 b a^{-2})^{m_1} \; \cdots \; (a b a^{-1})^{n_k} \; (a^2 b a^{-2})^{m_k}
\\ &= a \; b^{n_1} \; a \; b^{m_1} \; a^{-1} \; b^{n_2} \; a \; b^{m_2} \; a^{-1} \; \cdots  \; a^{-1} \; b^{n_k} \; a \; b^{m_k} \; a^{-2}
\end{align*}
and thus, using \eqref{eq.norm-cg},
$$\|c_g\|_2^2 = \frac{1}{100} \Bigl(2k +2 + \sum_{i=1}^k (n_i + m_i) \Bigr) + \frac{2k-1}{50} \; .$$
The same computation as above shows that $\Lambda \actson (X,\mu)$ is conservative. As above, it follows that $Q \in L^\infty(X_{\Lambda^c} \times \R)$. Altogether, we have proved that $Q \in 1 \ot L^\infty(\R)$.

So we get that $G \actson (X,\mu)$ is ergodic. To prove that the action is of type III$_1$, it suffices to show that the essential range of the map $x \mapsto \om(a,x)$ generates a dense subgroup of $\R_*^+$. But using \eqref{eq.RN-conv-ae}, we get that
$$\om(a,x) = \prod_{g \in G} \frac{\mu_{ag}(x_g)}{\mu_g(x_g)} = \frac{\mu_a(x_e)}{\mu_e(x_e)} = \begin{cases} 6/5 &\;\;\text{if $x_e = 0$,}\\ 4/5 &\;\;\text{if $x_e = 1$.}\end{cases}$$
Since $6/5$ and $4/5$ generate a dense subgroup of $\R_*^+$, the proposition is proved.


\end{proof}

\begin{remark}\label{rem.stable-examples-weird}
The stable type of the nonsingular Bernoulli action $\F_2 \actson (X,\mu)$ constructed in Proposition \ref{prop.concrete-free-group-1} is given as follows. The essential ranges of the maps $(x,x') \mapsto \om(g,x) / \om(g,x')$, $g \in \F_2$, generate the subgroup $(2/3)^\Z$ of $\R_*^+$ and $\om(g,x) \in (4/5) \cdot (2/3)^\Z$ for all $g \in \F_2$ and a.e.\ $x \in X$. Combining the proofs of Proposition \ref{prop.free-group-prescribed-invariants} and \ref{prop.concrete-free-group-1}, it follows that for every ergodic pmp action $\F_2 \actson (Y,\eta)$, the diagonal action $\F_2 \actson Y \times X$ is ergodic and that, varying $\F_2 \actson (Y,\eta)$, the type of this diagonal action ranges over III$_\mu$ with $\mu \in \{1\} \cup \{(2/3)^{1/k} \mid k \geq 1\}$.

Taking a slight variant of the action in Proposition \ref{prop.concrete-free-group-1}, by putting
$$\mu_g(0) = \begin{cases} 3/5 &\;\;\text{if $g \in W_a$,}\\ 5/12 &\;\;\text{if $g \in W_b$,} \\ 1/2 &\;\;\text{if $g \in W$,}\end{cases}$$
all the conclusions of Proposition \ref{prop.concrete-free-group-1} remain valid~--~except that we have to take $m \geq 317$ to get a dissipative diagonal action $\F_2 \actson X^m$~--~and moreover, the Maharam extension of $\F_2 \actson (X,\mu)$ is weakly mixing, so that all diagonal actions $\F_2 \actson Y \times X$ have type III$_1$. This follows because now, the essential ranges of the maps $(x,x') \mapsto \om(g,x) / \om(g,x')$, $g \in \F_2$, generate a dense subgroup of $\R_*^+$, namely the subgroup generated by $2/3$ and $5/7$.
\end{remark}

The Bernoulli action $\F_2 \actson (X,\mu)$ constructed in Proposition \ref{prop.concrete-free-group-1} has the property that the diagonal action $\F_2 \actson (X^m,\mu^m)$ is dissipative for $m$ large enough. This diagonal action is a Bernoulli action associated with $\F_2 \actson I$, where $I$ consists of $m$ disjoint copies of $\F_2$. This operation multiplies $\|c_g\|_2^2$ with a factor $m$, up to the point of satisfying the dissipative criterion in Proposition \ref{prop.conservative}. It is however remarkably more delicate to produce a plain Bernoulli action $\F_2 \actson \prod_{g \in \F_2} (\{0,1\},\mu_g)$ that is dissipative. We do this in the next result, based on Lemma \ref{lem.Z-special-cocycle} below, which provides a $1$-cocycle for $\Z$ with large growth, but bounded ``implementing function''.

\begin{proposition}\label{prop.free-group-dissipative}
Let $G = \F_2$. There exists a function $F : G \recht [1/4,3/4]$ such that the Bernoulli action $G \actson (X,\mu) = \prod_{g \in G} (\{0,1\},\mu_g)$ with $\mu_g(0) = F(g)$ is nonsingular, essentially free and dissipative.
\end{proposition}
\begin{proof}
Denote by $E_a \subset G$ the set of reduced words that end with a nonzero power of $a$. Similarly define $E_b$ and note that $G = \{e\} \sqcup E_a \sqcup E_b$. An element $g \in E_a$ is either a nonzero power of $a$ or can be uniquely written as $g = h a^n$ with $h \in E_b$ and $n \in \Z \setminus \{0\}$. We can therefore define
$$\pi_a : E_a \recht \Z : \pi_a(a^n) = n \;\;\text{when $n \in \Z \setminus \{0\}$, and}\;\; \pi_a(h a^n) = n \;\;\text{when $h \in E_b$ and $n \in \Z \setminus \{0\}$.}$$
We similarly define $\pi_b : E_b \recht \Z$.

Fix $D > 0$ such that $D > 32 \, \log 3$. Using Lemma \ref{lem.Z-special-cocycle}, fix a function $H : \Z \recht [0,1]$ such that $H(n) = 0$ for all $n \leq 0$ and such that the formula $\gamma_k(n) = H(n) - H(n-k)$ defines a $1$-cocycle $\gamma : \Z \recht \ell^2(\Z)$ satisfying $\|\gamma_k\|_2^2 \geq D |k|$ for all $k \in \Z$.

We define
$$F : G \recht [1/4,3/4] : F(g) = \begin{cases} 1/2 + H(\pi_a(g)) / 4 &\;\;\text{if $g \in E_a$,} \\ 1/2 -  H(\pi_b(g)) / 4 &\;\;\text{if $g \in E_b$,} \\ 1/2 &\;\;\text{if $g = e$.}\end{cases}$$
Define $c_g(h) = F(h) - F(g^{-1} h)$. Define the isometries
$$\theta_a : \ell^2(\Z) \recht \ell^2(G) : \theta_a(\delta_n) = \delta_{a^n} \quad\text{and}\quad \theta_b : \ell^2(\Z) \recht \ell^2(G) : \theta_b(n) = \delta_{b^n} \; .$$
We then have $c_a = \theta_a(\gamma_1) / 4$ and $c_b = - \theta_b(\gamma_1) / 4$. So, $c_g \in \ell^2(G)$ for every $g \in G$. It follows that the Bernoulli action $G \actson (X,\mu) = \prod_{g \in G} (\{0,1\},\mu_g)$ with $\mu_g(0) = F(g)$ is nonsingular and essentially free.

We prove that $\sum_{g \in G} \exp(- \|c_g\|_2^2 / 2) < \infty$. It then follows from Proposition \ref{prop.conservative} that $G \actson (X,\mu)$ is dissipative.

When
$$g = a^{n_0} \; b^{m_1} \; a^{n_1} \; \cdots \; b^{m_k} \; a^{n_k}$$
is a reduced word, with $k \geq 0$, $n_0,n_k \in \Z$ and $n_i,m_j \in \Z \setminus \{0\}$, the $1$-cocycle relation implies that
$$4 c_g = \theta_a(\gamma_{n_0}) - a^{n_0} \cdot \theta_b(\gamma_{m_1}) + a^{n_0} b^{m_1} \cdot \theta_a(\gamma_{n_1}) - \cdots + a^{n_0} b^{m_1} a^{n_1} \cdots b^{m_k} \cdot \theta_a(\gamma_{n_k}) \; .$$
All terms in the sum on the right hand side are orthogonal, except possibly consecutive terms, whose scalar products are equal to
$$- \langle \theta_a(\gamma_{n_i}), a^{n_i} \cdot \theta_b(\gamma_{m_{i+1}}) \rangle = - \gamma_{n_i}(n_i) \; \overline{\gamma_{m_{i+1}}(0)} = H(n_i) \, H(-m_{i+1}) \geq 0 \; ,$$
or equal to
$$- \langle \theta_b(\gamma_{m_i}), b^{m_i} \cdot \theta_a(\gamma_{n_i}) \rangle = - \gamma_{m_i}(m_i) \; \overline{\gamma_{n_i}(0)} = H(m_i) \, H(-n_i) \geq 0 \; .$$
We conclude that
$$16 \, \|c_g\|_2^2 \geq \sum_{i=0}^k \|\gamma_{n_i}\|_2^2 + \sum_{j=1}^k \|\gamma_{m_j}\|_2^2 \geq D \, \sum_{i=0}^k |n_i| + D \, \sum_{j=1}^k |m_j| = D \, |g| \; ,$$
where $|g|$ denotes the word length of $g \in \F_2$. So we have proved that $\|c_g\|_2^2 \geq (D / 16) \, |g|$ for all $g \in G$.

Since for $n \geq 1$, there are precisely $4 \cdot 3^{n-1}$ elements in $\F_2$ with word length equal to $n$, it follows that
$$\sum_{g \in G} \exp(- \|c_g\|_2^2 / 2) \leq \sum_{g \in G} \exp( - D \, |g| / 32 ) = 1 + 4 \, \sum_{n=1}^\infty \exp( - D \, n / 32 ) \; 3^{n-1} < +\infty \; ,$$
because $D/32 > \log 3$. So the proposition is proved.
\end{proof}

The function $H = 1_{[1,+\infty)}$ implements a $1$-cocycle $c : \Z \recht \ell^2(\Z)$ satisfying $\|c_k\|_2^2 = |k|$ for all $k \in \Z$. Multiplying $H$ by a constant $D > 0$, we obviously obtain a $1$-cocycle $c$ with growth $\|c_k\|_2^2 = D^2 \, |k|$. It is however more delicate to attain this growth while keeping $\|H\|_\infty \leq 1$. In particular, the easy construction of Lemma \ref{lem.translate-function} does not give such large growth. We need a more intricate construction with an oscillating function $H$, giving examples where $\|c_k\|_2^2 \geq D \, |k|^{3/2}$, while $H : \Z \recht [0,1]$.

\begin{lemma}\label{lem.Z-special-cocycle}
Let $D > 0$. There exists a function $H : \Z \recht [0,1]$ such that $H(n) = 0$ for all $n \leq 0$ and such that the formula $c_k(n) = H(n) - H(n-k)$ defines a $1$-cocycle $c : \Z \recht \ell^2(\Z)$ satisfying $\|c_k\|_2^2 \geq D |k|^{3/2}$ for all $k \in \Z$.
\end{lemma}
\begin{proof}
For every integer $n \geq 1$, define the function
$$H_n : \Z \recht [0,1] : H_n(k) = \begin{cases} k/n &\;\;\text{if $0 \leq k \leq n$,}\\ (2n-k)/n &\;\;\text{if $n \leq k \leq 2n$,} \\ 0 &\;\;\text{elsewhere.}\end{cases}$$
Let $(a_n)_{n \geq 0}$ be an increasing sequence of integers with $a_n \geq 1$ for all $n$ and $\sum_{n=0}^\infty a_n^{-1} < +\infty$. A concrete sequence $a_n$ will be chosen below. Put $b_0 = 0$ and $b_n = \sum_{k=0}^{n-1} 2a_k$ for all $n \geq 1$. Define the function
$$H : \Z \recht [0,1] : H(k) = \begin{cases} H_{a_n}(k-b_n) &\;\;\text{if $n \geq 0$ and $b_n \leq k \leq b_n + 2a_n$,} \\ 0 &\;\;\text{elsewhere.}\end{cases}$$
Note that we can view $H$ as a ``concatenation'' of translates of $H_{a_n}$, in such a way that their supports become disjoint. By construction, $H(k) = 0$ for all $k \leq 0$.

Define $c_k(n) = H(n) - H(n-k)$. We have
\begin{align*}
\|c_1\|_2^2 &= \sum_{m=1}^\infty |H(m) - H(m-1)|^2 = \sum_{n=0}^\infty \;\; \sum_{m=b_n+1}^{b_n + 2a_n} \;\; |H(m) - H(m-1)|^2 \\
& = \sum_{n=0}^\infty \;\; \sum_{m=b_n+1}^{b_n + 2a_n} \;\; \frac{1}{a_n^2} = 2 \sum_{n=0}^\infty \frac{1}{a_n} < +\infty \; .
\end{align*}
So, $c_1 \in \ell^2(\Z)$. Since $c$ satisfies the $1$-cocycle relation, we have that $c_k \in \ell^2(\Z)$ for all $k \in \Z$.

For every $k \geq 1$, define $\cF_k = \{n \in \Z \mid n \geq 0 \;\;\text{and}\;\; a_n \geq k\}$. For $k \geq 1$, we then have
\begin{align*}
\|c_k\|_2^2 & \geq \sum_{n \in \cF_k} \;\; \sum_{m = b_n+k}^{b_n+a_n} \;\; |c_k(m)|^2 = \sum_{n \in \cF_k} \;\; \sum_{m = b_n+k}^{b_n+a_n} \frac{k^2}{a_n^2} \\
& = \sum_{n \in \cF_k} \;\; \frac{k^2 (a_n - k + 1)}{a_n^2} \geq \frac{k^2}{2} \; \sum_{n \in \cF_{2k}} \;\; \frac{1}{a_n} \;\; ,
\end{align*}
where the last inequality follows because $\cF_{2k} \subset \cF_{k}$ and $a_n -k + 1 \geq a_n/2$ when $n \in \cF_{2k}$.

Let $D > 0$. Take $0 < \delta \leq 1$ such that $12 \sqrt{\delta} \leq D^{-1}$. Put $a_0 = 1$ and $a_n = \lceil \delta n^2 \rceil$ for all $n \geq 1$. We prove that $\|c_k\|_2^2 \geq D |k|^{3/2}$ for all $k \in \Z$. Since $\|c_{-k}\|_2 = \|c_k\|_2$, it suffices to prove this inequality for every $k \geq 1$.

Fix $k \geq 1$ and put $n_0 = \Bigl\lceil \sqrt{2k/\delta} \Bigr\rceil$. Note that $n_0 \geq 1$ and $\sqrt{\delta} n_0 \geq \sqrt{2k} \geq 1$. When $n \geq n_0$, we have $a_n \geq 2k$ and thus, $n \in \cF_{2k}$. Therefore,
\begin{align*}
\|c_k\|_2^2 & \geq \frac{k^2}{2} \; \sum_{n=n_0}^\infty \; \frac{1}{a_n} \geq \frac{k^2}{2} \; \sum_{n=n_0}^\infty \; \frac{1}{1 + \delta n^2} \\
& \geq \frac{k^2}{2} \; \int_{n_0}^\infty \; \frac{1}{1 + \delta x^2} \; dx = \frac{k^2}{2 \sqrt{\delta}} \; \bigl( \frac{\pi}{2} - \arctan(\sqrt{\delta} n_0)\bigr) \; .
\end{align*}
Since $\sqrt{\delta} n_0 \geq 1$ and $\frac{\pi}{2} - \arctan(x) \geq 1/(2x)$ for all $x \geq 1$, we get that
$$\|c_k\|_2^2 \geq \frac{k^2}{4 \delta n_0} \geq \frac{k^2}{4 \delta (\sqrt{2k/\delta} + 1)} = \frac{k^{3/2}}{4 \sqrt{\delta}} \, \frac{1}{\sqrt{2} + \sqrt{\delta/k}} \geq  \frac{k^{3/2}}{12 \sqrt{\delta}} \geq D \; k^{3/2} \; , $$
because $\sqrt{\delta / k} \leq 1$ and $12 \sqrt{\delta} \leq D^{-1}$.
\end{proof}

\section{\boldmath Amenable weakly mixing actions of stable type III$_\lambda$}

In this section, we give a positive answer to \cite[Question 4.6]{BN11} and prove the following result. The proof is independent of the rest of this article, but the result fits well with the above discussions on the stable type of nonsingular Bernoulli actions.

\begin{proposition}\label{prop.BN-question}
Let $G$ be an arbitrary countable infinite group and let $\lambda \in (0,1]$. Then $G$ admits an essentially free, amenable, weakly mixing action of stable type III$_\lambda$.
\end{proposition}

We first prove the proposition for infinite amenable groups, in particular for $\Z$, and then use an induction procedure to arbitrary infinite groups along a weakly mixing cocycle introduced in \cite{BN13}.

\begin{proof}
First consider the specific group $G_1 = \Z / 3 \Z \wr \Z = \bigl(\oplus_{n \in \Z} \Z / 3\Z \bigr) \rtimes \Z$ and let $\mu_0$ be a non uniform probability measure on $\Z / 3 \Z$ with $\mu_0(i) > 0$ for every $i \in \Z / 3 \Z$. Define $(X,\mu) = (\Z / 3 \Z,\mu_0)^\Z$ and consider the nonsingular action $G_1 \actson (X,\mu)$ where each $\Z / 3 \Z$ acts by translation on $\Z / 3 \Z$ and where $\Z$ acts by Bernoulli shift. If the ratios $\mu_0(0)/\mu_0(1)$ and $\mu_0(1)/\mu_0(2)$ generate a dense subgroup of $\R_*^+$, put $\lambda = 1$ and otherwise define $\lambda \in (0,1)$ so that this subgroup is given by $\lambda^\Z$. Since the action of $\Z$ on $(X,\mu)$ is pmp and weakly mixing and since $G_1$ is an amenable group, the action $G_1 \actson (X,\mu)$ is essentially free, amenable, weakly mixing action and of stable type III$_\lambda$.

Combining \cite{Fu99,BN13}, we say that $G_1$ is a weakly mixing measure equivalence (ME) subgroup of $G$ if there exists a $\sigma$-finite measure space $(\Omega,\nu)$ and an essentially free measure preserving action $G_1 \times G \actson (\Omega,\nu)$ with the following properties.
\begin{itemlist}
\item Both the restriction of the action to $G_1$ and the restriction of the action to $G$ are dissipative.
\item The restriction of the action to $G$ admits a fundamental domain of finite measure.
\item For every ergodic pmp action $G \actson (Y,\eta)$, the induced action $G_1 \actson (\Omega \times Y)/G$ is ergodic. Here, $G$ acts diagonally on $\Omega \times Y$ and since this action is dissipative, the quotient is well defined.
\end{itemlist}
By \cite{OW80}, all essentially free ergodic pmp actions of infinite amenable groups are orbit equivalent. Applying this to a pmp Bernoulli action of $G_1$, it follows that any infinite amenable group $G_1$ is a weakly mixing ME subgroup of any other infinite amenable group $G$. At the end of the proof, we use \cite{BN13} to prove that the group $\Z$ is a weakly mixing ME subgroup of any countable nonamenable group $G$. So to conclude the proof of the proposition, it suffices to prove the following statement: if $G_1$ is a weakly mixing ME subgroup of $G$ and if $G_1$ admits an essentially free, amenable, weakly mixing action of stable type III$_\lambda$, then also $G$ admits such an action.

Take $(\Omega,\nu)$ as above and let $G_1 \actson (Z,\eta)$ be an essentially free, amenable, weakly mixing action of stable type III$_\lambda$. Consider the essentially free nonsingular action $G \actson (Z \times \Omega)/G_1$, where $G_1$ acts diagonally on $Z \times \Omega$. Fix an ergodic pmp action $G \actson (Y,\rho)$. We have to prove that the diagonal action $G \actson (Z \times \Omega)/G_1 \times Y$ is amenable, ergodic and of type III$_\lambda$. The corresponding orbit equivalence relation is isomorphic to the restriction of the orbit equivalence relation of $G_1 \times G \actson Z \times \Omega \times Y$ to a non negligible subset. So, the diagonal action $G \actson (Z \times \Omega)/G_1 \times Y$ is stably orbit equivalent to the diagonal action $G_1 \actson Z \times (\Omega \times Y)/G$. We therefore have to prove that the latter is amenable, ergodic and of type III$_\lambda$. The amenability follows because $G_1 \actson Z$ is amenable. Since $G_1 \actson(\Omega \times Y)/G$ is ergodic and pmp and since $G_1 \actson Z$ is weakly mixing and of stable type III$_\lambda$, we get that $G_1 \actson Z \times (\Omega \times Y)/G$ is ergodic and of type III$_\lambda$.

It remains to prove that $\Z$ is a weakly mixing ME subgroup of any countable nonamenable group $G$. Fix a symmetric probability measure $\mu_0$ on $G$ whose support generates the group $G$. Put $(X,\mu) = (G,\mu_0)^\Z$ and let $\Z \actson (X,\mu)$ be the Bernoulli shift, given by $(n \cdot x)_k = x_{k - n}$. Denote by $\om : \Z \times X \recht G$ the $1$-cocycle introduced in \cite[Theorem 6.1]{BN13} and uniquely determined by $\om(1,x) = x_0$. Denote by $\lambda$ the counting measure on $G$ and define $(\Omega,\nu) = (X \times G,\mu \times \lambda)$. Define the action $\Z \times G \actson \Omega$ given by
$$(n,g) \cdot (x,h) = (n\cdot x, \om(n,x) h g^{-1}) \quad\text{for all}\;\; (n,g) \in \Z \times G \; , \; (x,h) \in X \times G \; .$$
This action is essentially free and measure preserving. Also, the restriction of the action to $G$ has $X \times \{e\}$ as a finite measure fundamental domain. In \cite[Theorem 6.1]{BN13}, it is proven that $\Z \actson (\Omega \times Y)/G$ is ergodic for every pmp ergodic action $G \actson (Y,\eta)$. So we only have to prove that the action $\Z \actson \Omega$ is dissipative.

Fix $g_0 \in G$ and define
$$V_{g_0} = \{(x,g) \in \Omega \mid \forall k \geq 0 : x_{-k} \cdots x_{-1} x_0 g \neq g_0 \} \; .$$
Defining $\pi : \Omega \recht G : \pi(x,g) = g$, we have that
$$V_{g_0} = \{(x,g) \in \Omega \mid \forall k \geq 1 : \pi(k \cdot (x,g)) \neq g_0 \} \; .$$
So, $1 \cdot V_{g_0} \subset V_{g_0}$.

For every fixed $g \in G$, the measure
\begin{equation}\label{eq.mymeasure}
\mu\bigl(\{x \in X \mid \; \text{there are infinitely many $k \geq 0$ with $x_{-k} \cdots x_{-1} x_0 g = g_0$}\;\}\bigr)
\end{equation}
equals the probability that the invariant random walk on $G$ with transition probabilities given by $\mu_0$ and starting at $g$ visits infinitely often the element $g_0$. Since the group $G$ is nonamenable and the support of $\mu_0$ generates $G$, this random walk is transient and the measure in \eqref{eq.mymeasure} is zero for every $g \in G$. This means that $\bigcup_{k \in \Z} k \cdot V_{g_0}$ has a complement of measure zero for every $g_0 \in G$. Since $1 \cdot V_{g_0} \subset V_{g_0}$, it follows that the action of $\Z$ on $\Omega_{g_0} = \Omega \setminus \bigcap_{k \in \Z} k \cdot V_{g_0}$ is dissipative with fundamental domain $V_{g_0} \setminus 1 \cdot V_{g_0}$. Since $X \times \{g_0\} \subset \Omega_{g_0}$ and $g_0 \in G$ is arbitrary, it follows that $\Z \actson \Omega$ is dissipative.
\end{proof}

\begin{remark}
For a countable nonamenable group $G$, the proof of Proposition \ref{prop.BN-question} provides an explicit essentially free, amenable, weakly mixing action of stable type III$_1$. Indeed, it suffices to combine the explicit action $\Z \actson (Z,\eta)$ of stable type III$_1$ given by Corollary \ref{cor.explicit-Z} with the explicit $1$-cocycle $\om : \Z \times X \recht G$ of \cite[Theorem 6.1]{BN13}.

Note that the resulting amenable, weakly mixing and stable type III$_1$ action of $G$ on $\Xi = (Z \times X \times G)/\Z$ has the property that the diagonal action $G \actson \Xi \times \Xi$ is dissipative, contrary to the action of $G$ on its Poisson boundary, which is doubly ergodic. To prove that $G \actson \Xi \times \Xi$ is dissipative, we write $\Lambda = \Z \times \Z$ and note that it is sufficient to prove that the action of $\Lambda$ on $(Z \times Z \times X \times X \times G \times G)/G$ is dissipative. So, it suffices to prove that $\Lambda \actson (X \times X \times G \times G)/G$ is dissipative. This means that we have to prove that the action
$$\Lambda \actson X \times X \times G \quad\text{given by}\quad (k,l) \cdot (x,y,g) = (k\cdot x,l \cdot y, \om(k,x) g \om(l,y)^{-1})$$
for all $(k,l) \in \Z^2$ and $(x,y,g) \in X \times X \times G$, is dissipative.

For all $(k,l) \in \Lambda$, denote
$$V_{k,l} = \{(x,y,e) \in X \times X \times G \mid \om(k,x) \neq \om(l,y) \} \; .$$
For every $n \geq 1$, denote $\Lambda_n = n \Z \times n \Z$ and write
$$V_n = \bigcap_{(k,l) \in \Lambda_n \setminus \{(0,0)\}} V_{k,l} \; .$$
For $k,l \geq 1$, we have
\begin{align*}
V_{k,l} &= \{(x,y,e) \mid x_{-k+1} \cdots x_{-1} x_0 \neq y_{-l+1} \cdots y_{-1} y_0 \} \\ &= \{(x,y,e) \mid x_{-k+1} \cdots x_0 y_0^{-1} \cdots y_{-l+1}^{-1} \neq e\} \; ,
\end{align*}
so that $(\mu \times \mu \times \lambda)(V_{k,l}) = 1-\mu^{*(k+l)}(e)$. When $k \geq 1$ and $l \leq -1$, we similarly find
$$V_{k,l} = \{(x,y,e) \mid x_{-k+1} \cdots x_0 y_{-l}^{-1} \cdots y_1^{-1} \neq e\}$$
and conclude that $(\mu \times \mu \times \lambda)(V_{k,l}) = 1-\mu^{*(|k|+|l|)}(e)$ for all $k,l \in \Z$. Since $G$ is nonamenable, we can fix $0 < \rho < 1$ so that $\mu^{*m}(e) \leq \rho^m$ for all $m \geq 1$. It follows that $(\mu \times \mu \times \lambda)(V_n) \recht 1$, so that $\bigcup_{n \geq 1} V_n$ equals $X \times X \times \{e\}$, up to measure zero.

By construction, $(k,l) \cdot V_n \cap (X \times X \times \{e\}) = \emptyset$ for all $(k,l) \in \Lambda_n \setminus \{(0,0)\}$. In particular, $(k,l) \cdot V_n \cap V_n = \emptyset$, so that the action $\Lambda_n \actson \Lambda_n \cdot V_n$ is dissipative. Since $\Lambda_n < \Lambda$ has finite index, also the action $\Lambda \actson \Lambda \cdot V_n$ is dissipative. Since the union of all $V_n$ equals $X \times X \times \{e\}$ up to measure zero, we conclude that the action $\Lambda \actson \Lambda \cdot (X \times X \times \{e\}) = X \times X \times G$ is dissipative.
\end{remark}

\begin{remark}
Generalizing the action of the wreath product group $\Z / 3 \Z \wr \Z$ that we used in the beginning of the proof Proposition \ref{prop.BN-question}, we can also provide a negative answer to \cite[Problem H]{Mo06}. Let $\Gamma$ and $\Lambda$ be countable groups with $\Gamma$ nonamenable and $\Lambda$ infinite amenable. Define $G = \Gamma \wr \Lambda = \Gamma^{(\Lambda)} \rtimes \Lambda$. Choose a nonsingular amenable action $\Gamma \actson (X_0,\mu_0)$. Define $(X,\mu) = (X_0,\mu_0)^\Lambda$ and consider the action $G \actson (X,\mu)$, where $\Lambda$ acts by Bernoulli shift and where each copy of $\Gamma$ acts on the corresponding copy of $(X_0,\mu_0)$ in the infinite product. We get that $G \actson (X,\mu)$ is amenable and that all its power actions $G \actson X^n = X \times \cdots \times X$ are ergodic, because the restriction of $G \actson (X,\mu)$ to the subgroup $\Lambda$ is a pmp Bernoulli action. So $G$ is a nonamenable group with infinite amenability degree, in the sense of \cite[Definition 3.2]{Mo06}. Therefore, $G$ provides a negative answer to \cite[Problem H]{Mo06}. It similarly follows that the bounded cohomology of $G$ with coefficients in an arbitrary semi-separable coefficient $G$-module $V$ (in the sense of \cite[Definition 3.11]{Mo07}) vanishes in all degrees: $H^0_{\text{\rm b}}(G,V) = V^G$ and $H^n_{\text{\rm b}}(G,V) = \{0\}$ for all $n \geq 1$.
\end{remark}

\end{document}